\documentclass[12pt]{amsart}

\usepackage{fullpage}  
\usepackage[pagewise]{lineno}

\usepackage[utf8]{inputenc}

\usepackage{amsmath}  
\usepackage{amsthm}  
\usepackage{amssymb}  
\usepackage{bm}  
\usepackage{mathtools}  

\usepackage{enumerate}

\usepackage{graphicx}  
\usepackage[dvipsnames]{xcolor}  
\usepackage[small,bf]{caption}  

\graphicspath{{figures/}}  
\graphicspath{{../figures/}}  

\usepackage{breakcites}  
\usepackage{url}  
\usepackage{hyperref}  

\hypersetup{colorlinks, citecolor=black, filecolor=black,
	linkcolor=black, urlcolor=black}

\usepackage{algorithm}
\usepackage[noend]{algpseudocode}

\usepackage{comment}  

\theoremstyle{plain}

\newtheorem{outtheorem}{Theorem}

\newtheorem{theorem}{Theorem}[section]
\newtheorem{proposition}[theorem]{Proposition}
\newtheorem{lemma}[theorem]{Lemma}
\newtheorem{corollary}[theorem]{Corollary}

\theoremstyle{definition}

\newtheorem{definition}[theorem]{Definition}
\newtheorem{exampleEnv}[theorem]{Example}
{\pushQED{\qed}\exampleEnv}
{\popQED\endexampleEnv}
\newenvironment{example*}{\exampleEnv}{\endexampleEnv}
\newtheorem{remarkEnv}[theorem]{Remark}
{\pushQED{\qed}\remarkEnv}
{\popQED\endremarkEnv}
\newenvironment{remark*}{\remarkEnv}{\endremarkEnv}

\def\gg{\mathfrak{g}}
\def\hh{\mathfrak{h}}
\def\mm{\mathfrak{m}}
\def\kk{\mathfrak{k}}
\def\ll{\mathfrak{l}}
\def\pp{\mathfrak{p}}
\def\zz{\mathfrak{z}}

\usepackage{scalerel}
\usepackage{stackengine,wasysym}

\newcommand\reallywidetilde[1]{\ThisStyle{%
		\setbox0=\hbox{$\SavedStyle#1$}%
		\stackengine{-.1\LMpt}{$\SavedStyle#1$}{%
			\stretchto{\scaleto{\SavedStyle\mkern.2mu\AC}{.5150\wd0}}{.6\ht0}%
		}{O}{c}{F}{T}{S}%
}}

\usepackage{tikz-cd}
\usepackage{graphicx}
\graphicspath{ {./images/} }

\numberwithin{equation}{section}

\DeclareMathOperator{\ric}{ric}
\DeclareMathOperator{\Ad}{Ad}
\DeclareMathOperator{\tr}{Tr}
\DeclareMathOperator{\ad}{ad}

\DeclareMathOperator{\Rm}{Rm}

\makeatletter
\@namedef{subjclassname@2020}{%
	\textup{2020} Mathematics Subject Classification}
\makeatother

\title{\bf Long-time behavior of awesome homogeneous Ricci flows}
\author{Roberto Araujo}
\address{Universit\"at M\"unster, Mathematisches Institut\\
	Einsteinstr. 62\\
	48149 M\" unster\\
	Germany}
\email{r.araujo@uni-muenster.de}
\thanks{Funded by the Deutsche Forschungsgemeinschaft (DFG, German Research Foundation) under Germany’s Excellence Strategy EXC 2044–390685587, Mathematics Münster: Dynamics–Geometry–Structure and the CRC 1442 Geometry: Deformations and Rigidity}

\begin{document}

\begin{abstract}
We show that the set of \textit{awesome} homogeneous metrics on non-compact manifolds is Ricci flow invariant. Moreover, if the universal cover of such awesome homogeneous space is not contractible the Ricci flow has finite extinction time, confirming the Dynamical Alekseevskii Conjecture in this case. We also analyze the long-time limits of awesome homogeneous Ricci flows.
\end{abstract}

\maketitle


\section{Introduction} \label{section:1}

The Ricci flow is the geometric evolution equation on given by 

\begin{equation*}\label{eq:general RF}
	\frac{\partial g(t)}{\partial t} = -2\ric(g(t)) , \hspace{0.5cm} g(0)=g_0,
\end{equation*}
where $\text{ric}(g)$ is the Ricci $(2,0)$-tensor of the Riemannian manifold $(M,g)$.

Hamilton introduced the Ricci flow in \cite{ham} and proved short time existence and uniqueness when $M$ is compact. Then Chen and Zhu \cite{chenzu} proved the uniqueness of the flow within the class of complete and bounded curvature Riemannian manifolds. 

A maximal Ricci flow solution $g(t)$, $t \in[0,T)$ is called \textit{immortal} if $T=+\infty$, otherwise we say that the flow has \textit{finite extinction time}.

A Riemannian manifold $(M,g)$ is called \textit{homogeneous} if its isometry group acts transitively on it. From the uniqueness of a Ricci flow solution it follows immediately that the isometries are preserved along the flow, thus a solution $g(t)$ from a homogeneous initial metric $g_0$ would remain homogeneous for the same isometric action. Hence, the Ricci flow equation given above becomes an autonomous nonlinear ordinary differential equation. 

 In the homogeneous case, the scalar curvature is increasing along the flow (see \cite{laf}). Furthermore, if the scalar curvature is positive at some point along the flow, then it must blow up in finite time and hence, the solution is not immortal. Lafuente \cite{laf} has shown that actually a homogeneous Ricci flow solution has finite extinction time if and only if the scalar curvature blows up in finite time, or equivalently, if and only if the scalar curvature ever becomes positive along the flow. Moreover, Bérard-Bergery \cite{ber} has shown that a manifold admits a homogeneous Riemannian metric of positive scalar curvature if and only if its universal cover is not diffeomorphic to Euclidean space. 

 B\"ohm and Lafuente then proposed on \cite{bl18} the problem of showing whether the converse is also true, namely they asked whether the universal cover of an immortal homogeneous Ricci flow solution is always diffeomorphic to $\mathbb{R}^n$. This got established later as the \textit{Dynamical Alekseevskii Conjecture} \cite{mfo22}.

Not much is known in the direction of the Dynamical Alekseevskii Conjecture other than in low-dimensions. In the work of Isenberg-Jackson \cite{ij}, the 3-dimensional homogeneous Ricci flow has been thoroughly studied, moreover in \cite{ijl} the authors study a large set of metrics on dimension 4. Indeed, up to dimension 4 the conjecture is true (see \cite{ara}). Moreover, in \cite{ara}, it was shown that the conjecture is true if the isometry group of the homogeneous Riemannian manifold is, up to a covering, a Lie group product with a compact semisimple factor.

In this work we study the long-time behavior of the homogeneous Ricci flow solutions on semisimple homogeneous spaces on a special family of Ricci flow invariant metrics, called \textit{awesome metrics}.

Let $G$ be a semsimple Lie group and $G/H$ a homogeneous Riemannian manifold. Let $\gg$ be the Lie algebra of $G$ and $\hh$ the Lie algebra of $H$. Let $\gg =\kk \oplus \pp$ be a Cartan decomposition of $\gg$, $K$ the integral subgroup corresponding to $\kk$, and $\mm=\ll \oplus\pp$ a reductive complement to  $\hh$. Then we call a $G$-invariant metric $g$ on $G/H$ \textit{awesome} if $g(\ll,\pp)=0$.

The set of awesome homogeneous metrics was introduced by Nikonorov \cite{n}, where he proved that it contains no Einstein metric. Semisimple homogeneous spaces with inequivalent irreducible summands in its isotropy representation supply the simplest examples of homogeneous spaces $G/H$ such that every $G$-invariant metric is awesome.

The first main result of this article is a generalization of the work of Nikonorov \cite[Theorem 1]{n} to the dynamical setting, giving a partial positive answer to the conjecture. 

\begin{outtheorem}\label{thm:A}
	Let $(G/H,g_0)$ be a homogeneous Riemannian manifold such that the universal cover is not diffeomorphic to $\mathbb{R}^n$ and $G$ is semisimple. If $g_0$ is an awesome $G$-invariant metric, then the Ricci flow solution starting at $g_0$ has finite extinction time.
\end{outtheorem}

In \cite{dl}, Dotti and Leite have shown that $SL(n,\mathbb{R})$ for $n\geq3$ admit left-invariant Ricci negative metrics. Later, Dotti, Leite and Miatello \cite{dlm} were able to extend this result by showing that all but a finite collection of non-compact simple Lie groups admit a Ricci negative left-invariant metric. All those metrics are awesome. This shows that for theses spaces such that the universal cover $\tilde{M}$ is not contractible the confirmation of the Dynamical Alekseevskii Conjecture implies a change of regime of the Ricci flow: from one in which the manifold expands in all directions to one such that for some direction it shrinks in finite time. B\"ohm's proof \cite[Theorem 3.2]{boe} of the finite extinction time of the Ricci flow on non-toral compact homogeneous manifolds works by showing an explicit preferred direction in which the curvature is Ricci positive (the same approach is followed in \cite{ara}), but Dotti, Leite and Miatello's results indicate that we cannot directly do the same here.

Indeed, in order to prove Theorem \ref{thm:A} we need to first prove some scale-invariant pinching estimates (Proposition \ref{prop:pinching}) that will eventually lead to the existence of a Ricci positive direction given by the non-toral compact fibers as in \cite{boe}. Moreover, the estimates obtained can be exploited further to prove the following two convergence results.

\begin{outtheorem}\label{thm:B}
	Let $M=G/H$ be a homogeneous manifold, such that the universal cover is not diffeomorphic to $\mathbb{R}^n$ and $G$ is semisimple. Let $(M,g(t))$, $t \in [0,T)$, be an awesome Ricci Flow adapted to the Cartan decomposition $\gg = \kk \oplus \pp$. Let $R(g)$ be the scalar curvature of the metric $g$. For any sequence $(t_a)_{a\in\mathbb{N}}$, $t_a \to T$ there exists a subsequence such that $\left(M,R(g(t_{\hat{a}}))\cdot g(t_{\hat{a}})\right)$ converges in pointed $C^{\infty}$-topology to the Riemannian product
	
	\[E_{\infty} \times \mathbb{E}^d,\]
	
	where $E_{\infty}$ is a compact homogeneous Einstein manifold with positive scalar curvature and $\mathbb{E}^d$ is the $d$-dimensional (flat) Euclidean space with $d \geq \dim \pp$.
	
		Furthermore, the geometry of $E_{\infty}$ just depends on the subsequence of Riemannian submanifolds $\left(K/H,R(g(t_{\hat{a}}))\cdot g(t_{\hat{a}})\right)$. 
\end{outtheorem}

\begin{outtheorem}\label{thm:C}
	Let $\tilde{M}=G/H$ be a homogeneous manifold diffeomorphic to $\mathbb{R}^n$ with $G$ semisimple. Let $(\tilde{M},g(t))$, $t \in [1,\infty)$, be an awesome Ricci Flow adapted to the Cartan decomposition $\gg = \kk \oplus \pp$. Then the parabolic rescaling $(\tilde{M},t^{-1}g(t))$ converges in pointed $C^{\infty}$-topology to the Riemannian product
	\[ \Sigma_{\infty}\times \mathbb{E}^{\dim \ll},\]
	where $\Sigma_{\infty}=\left(G/K,\left.B\right|_{\pp\times\pp}\right)$ is the non-compact Einstein symmetric space defined by the pair $(\gg,\kk)$ and $\mathbb{E}^d$ is the $d$-dimensional (flat) Euclidean space. 
\end{outtheorem}

Theorem \ref{thm:B} shows that in order to understand the blow-up limits of the Ricci flow on the awesome metrics we can reduce the investigation to the corresponding blow-up of the compact homogeneous fibers given by the Cartan decomposition. Such analysis was done for example in \cite{boe}. The other convergence result is the following.

Since every left-invariant metric on $\reallywidetilde{SL(2,\mathbb{R})}$ is awesome, Theorem \ref{thm:C} is a generalization of the result by Lott \cite{lot} which states that the parabolic blow-down of any left-invariant metric in $\reallywidetilde{SL(2,\mathbb{R})}$ converges to the Riemannian product $\mathbb{H}^2 \times \mathbb{R}$. 

The structure of this article is the following: In Section \ref{section:2}, we give a quick overview of the homogeneous Ricci flow and show that the space of awesome metrics is Ricci flow invariant. In Section \ref{section:3}, we mainly establish a priori algebraic bounds that exploit the compatibility of the Cartan decomposition and the metric in the awesome case. In Section \ref{section:4}, we use these algebraic bounds to get control quantities to our dynamics which allows us to prove theorem \ref{thm:A}. Finally, in Section \ref{section:5} we conclude with the analysis of the long-time limits. In particular, under the hypothesis of Theorem \ref{thm:A}, we show in Theorem \ref{thm:B} a rigidity result for the possible limit geometries as the solution approaches the singularity. We finish by showing in Theorem \ref{thm:C} which is the limit geometry at infinity for the case when $g(t)$ is an immortal awesome Ricci flow. This generalizes the work on the Ricci flow of left-invariant metrics on $\reallywidetilde{SL(2,\mathbb{R})}$ done in \cite{lot} to $\mathbb{R}^d$-bundles over Hermitian symmetric spaces.  

\textit{Acknowledgements.} I would like to thank my PhD advisor Christoph B\"ohm for his support and helpful comments, as well as to Ramiro Lafuente for the insightful conversations and fruitful comments and suggestions, in particular regarding the last section of this paper. 
I would also like to
thank the referee for the useful remarks.

\section{Homogeneous Ricci Flow of Awesome Metrics}\label{section:2}

A Riemannian manifold $(M^n, g)$ is said to be homogeneous if its isometry group
$I(M, g)$ acts transitively on $M$. If $M$ is connected (which we will assume from here onward unless otherwise stated), then each transitive closed Lie subgroup $G <I(M, g)$ gives
rise to a presentation of $(M, g)$ as a homogeneous space with a $G$-invariant metric
$(G/H, g)$, where $H$ is the isotropy subgroup of $G$ fixing some point $p \in M$. 

The $G$-action induces a Lie algebra homomorphism $\gg \to \mathfrak{X}(M)$ assigning to each $X \in \gg$ a Killing field on $(M, g)$, also denoted by X, and given by

\[X(q) \coloneqq \left. \left(\frac{d}{dt} \exp (tX)\cdot q\right) \right|_{t=0}, \hspace{0.5cm} q \in M.\]

If $\hh$ is the Lie algebra of the isotropy subgroup $H < G$ fixing $p \in M$, then it can be characterized as those $X \in \gg$ such that $X(p)=0$. Given that, we can take a complementary $\text{Ad}(H)$-module $\mm$ to $\hh$ in $\gg$ and identify $\mm \cong T_pM$ via the above infinitesimal correspondence.

A homogeneous space $G/H$ is called \textit{reductive} if there exists such a complementary vector space $\mm$ such that for the respective Lie algebras of $G$ and $H$ 

\begin{equation*}
	\gg = \hh \oplus \mm, \hspace{0.5cm} \text{Ad}(H)(\mm) \subset \mm
\end{equation*}

This is always possible in the case of homogeneous Riemannian manifolds. This is due to a classic result on Riemannian geometry \cite[Chapter VIII, Lemma 4.2]{doC}, which states that an isometry is uniquely determined by the image of the point $p$ and its derivative at $p$, hence the isotropy subgroup $H$ is a closed subgroup of $SO(T_pM)$, and in particular it is compact. Indeed, if $\Ad(H)$ is compact, then one can average over an arbitrary inner product over $\gg$ to make it Ad($H$)-invariant and hence take $\mm \coloneqq \hh^{\perp}$. Given that, one can identify $\mm \cong T_{eH}G/H$ once and for all and with this identification there is a one-to-one correspondence between homogeneous metrics in $M \coloneqq G/H$, $p \cong eH$ and Ad($H$)-invariant inner products in $\mm$. 

In full generality, the Ricci flow is a non-linear partial differential equation. As mentioned in the introduction, in the case where $M$ is compact, Hamilton \cite{ham} proved short time existence and uniqueness for the Ricci flow. Then Chen and Zhu \cite{chenzu} proved the uniqueness of the flow within the class of complete and bounded curvature Riemannian manifolds, which includes the class of homogeneous manifolds. From its uniqueness it follows immediately that the Ricci flow preserves isometries. Thus a solution $g(t)$ from a $G$-invariant initial metric $g_0$ would remain $G$-invariant, and hence it is quite natural to consider a \textit{homogeneous Ricci flow}. We then get an autonomous nonlinear ordinary differential equation, 

\begin{align}\label{eq:ricci flow}
	\frac{dg(t)}{dt} &= -2\text{ric}(g(t)), \hspace{0.5cm} g(0)=g_0,
\end{align}
where the Ricci tensor can being seen as the following smooth map 
\[\text{ric}\colon (\text{Sym}^2(\mm))^{\Ad(H)}_+ \to (\text{Sym}^2(\mm))^{\Ad(H)}.\]
Here $(\text{Sym}^2(\mm))^{\Ad(H)}$ is the nontrivial vector space of $\Ad(H)$-invariant symmetric bilinear forms in $\mm$ and $(\text{Sym}^2(\mm))^{\Ad(H)}_+$ the open set of positive definite ones. By classical ODE theory, given an initial $G$-invariant metric $g_0$ corresponding to an initial $\Ad(H)$-invariant inner product, there is an unique Ad($H$)-invariant inner product solution corresponding to an unique family of $G$-invariant metrics $g(t)$ in $M$. 

The general formula for the Ricci curvature of a homogeneous Riemannian manifold $(G/H,g)$, \cite[Corollary 7.38]{be}, is given by 

\begin{align}\label{eq:ricci formula}
	\text{ric}_g(X,X)=&- \frac{1}{2}B(X,X)-\frac{1}{2}\sum_i |[X,X_i]_{\mm}|_g^2  + \frac{1}{4}\sum_{i,j}g\left(\left[X_i,X_j\right]_{\mm},X\right)^2 \\
	&- g\left(\left[H_g,X\right]_{\mm},X\right), \notag
\end{align}
where $B$ is the Killing form, $\{X_i\}_{i=1}^n$ is a $g$-orthonormal basis of $\mm$ and $H_g$ is the mean curvature vector defined by $g(H_g,X) \coloneqq\tr(\text{ad}_X)$. Immediately it follows that $H_g=0$ if and only if $\gg$ is unimodular. 

Let us now consider $\gg$ to be a non-compact semisimple Lie algebra. By classical structure theory on semisimple Lie algebras \cite[Chapter 13]{hn}, $\gg$ can be described in terms of its \textit{Cartan decomposition}
\[\gg = \kk \oplus \pp,\]
where $\kk$ is a compactly embedded Lie subalgebra of $\gg$ and $\pp$ is a $\kk$-submodule such that $[\pp,\pp] \subset \kk$. 

Moreover, the Killing form $B$ of $\gg$ is such that

\begin{equation*}
	B(\kk,\pp) =0, \hspace{0.5cm} 
	B|_{\kk\times\kk} <0, \hspace{0.5cm}
	\left.B\right| _{\pp \times \pp} >0,
\end{equation*}
and $-B|_{\kk\times\kk}+B|_{\pp\times\pp}$ defines an inner product on $\gg$ such that $\ad(\kk)$ are skew-symmetric maps and $\ad(\pp)$ are symmetric maps \cite[Lemma 13.1.3]{hn}.

Since the flow only depends on the Lie algebra $\gg$, we can take without loss of generality any $G$ connected with such Lie algebra. So for a $M=G/H$, with $G$ a semisimple non-compact Lie group, we can fix a background Cartan decomposition 
\[\gg = \kk \oplus \pp\]
such that the integral subgroup $K$ of $\kk$ is a maximal connected compact subgroup of $G$ with $H \subset K$ \cite[Theorem 14.1.3]{hn}. We call a homogeneous manifold $G/H$, with $G$ semisimple and $\Ad(H)$ compact a \textit{semisimple homogeneous space}.

Consider the orthogonal complement $\ll:= \hh^{\perp}$ of $\hh$ in $\kk$ with respect to the Killing form $B$. And let us do the identification
\begin{equation} \label{eq:awesome cartan decomposition}
	T_{eH}G/H \cong \mm = \ll \oplus \pp.
\end{equation} 
We will call then this reductive complement $\mm =\ll \oplus \pp$ \textit{adapted} to the Cartan decomposition $\gg= \kk \oplus \pp $

\begin{definition}[\textit{Awesome metric}]
	Let $G/H$ be a homogeneous space with $G$ semisimple and $\Ad(H)$ compact. An $\Ad(H)$-invariant inner product $g$ on the reductive complement $\mm$ is called \textit{awesome} if for some Cartan decomposition $\gg= \kk \oplus \pp$ for which $\mm =\ll \oplus \pp$ is adapted, we have that $g(\ll, \pp)=0$. In this case, we say that the awesome metric $g$ is adapted to the Cartan decomposition $\gg= \kk \oplus \pp$, with $\kk=\hh\oplus\ll$.
\end{definition}

As it was mentioned in the introduction, this non-empty set of metrics was introduced by Nikonorov in \cite{n}, where he proved that the set contains no Einstein metric. We will see that this set is actually Ricci flow invariant, proving to be a good test ground concerning the Dynamical Alekseevskii Conjecture. Semisimple homogeneous spaces $G/H$ such that the isotropy representation of $H$ on $\mm$ have inequivalent irreducible summands only admit awesome $G$-invariant metrics, and as such this set of metrics has an obvious spotlight in the literature. 

On the other hand, for example in the case of a Lie group with dimension larger than $3$, the set of awesome  metrics is a meager subset of the left-invariant metrics, and its dynamical properties under the phase space of the Ricci flow are largely unknown. Nikonorov also gave a necessary and sufficient algebraic condition for a semisimple homogeneous space $G/H$ to be such that every $G$-invariant metric is awesome \cite[Theorem 2]{n}. 

Via \eqref{eq:awesome cartan decomposition}, we have an one-to-one correspondence between the set of awesome $G$-invariant metrics on $G/H$ and the  open subset of positive definite $\Ad(H)$-invariant symmetric bilinear forms in $\mm$ such that $\ll \perp \pp$, which in turn is a linear subspace of $(\text{Sym}^2(\mm))^{\Ad(H)}$.

Manipulating the Ricci tensor formula \eqref{eq:ricci formula} on the awesome case we can directly prove the following lemma.

\begin{lemma}\label{lemma:awesome is RF invariant}
	Let $G/H$ be a semisimple homogeneous space. Then the set of $G$-invariant awesome metrics in $G/H$ is Ricci flow invariant.
\end{lemma}

\begin{proof}
Let $G/H$ be a semisimple homogeneous space. We need to show that the Ricci operator $\text{ric}: (\text{Sym}^2(\mm))^{\Ad(H)}_+ \to (\text{Sym}^2(\mm))^{\Ad(H)}$ takes an element $g$ such that $g(\ll,\pp)=0$ to a symmetric bilinear form $\ric_g$ such that $\ric_g(\ll,\pp)=0$. By polarizing the formula for the Ricci tensor \eqref{eq:ricci formula} on a homogeneous manifold $G/H$ with $G$ unimodular, we get 

\begin{equation*}
	2\text{ric}_g(X,Y)=  - B(X,Y)-\sum_i g([X,X_i]_{\mm}, [Y,X_i]_{\mm}) + \frac{1}{2}\sum_{i,j}g([X_i,X_j]_{\mm},X)g([X_i,X_j]_{\mm},Y). 
\end{equation*}

In our case $\{X_1, ..., X_{n+m}\}= \{X^\ll_1,\ldots,X^\ll_n ,X^\pp_1,\ldots, X^\pp_m\}$, where $\{X^\ll_i\}_{i=1}^n$ and $\{X^{\pp}_i\}_{i=1}^m$ are $g$-orthonormal basis for $\ll$ and $\pp$ respectively. We then get that for $X^{\ll} \in \ll$ and $X^\pp \in \pp$ 

\begin{align*}
	2\text{ric}_g\left(X^{\ll},X^\pp\right)=&  - B\left(X^{\ll},X^\pp\right)
	-\sum_i g\left(\left[X^{\ll},X_i\right]_{\mm}, \left[X^\pp,X_i\right]_{\mm}\right) \\
	&+\frac{1}{2}\sum_{i,j}g\left(\left[X_i,X_j\right]_{\mm},X^{\ll}\right)g\left(\left[X_i,X_j\right]_{\mm},X^\pp\right) \\
	=&-\sum_i g\left(\left[X^{\ll},X^\ll_i\right]_{\mm}, \left[X^\pp,X^\ll_i\right]_{\mm}\right) -\sum_i g\left(\left[X^{\ll},X^{\pp}_i\right]_{\mm}, \left[X^\pp,X^{\pp}_i\right]_{\mm}\right) \\
	&+ \frac{1}{2}\sum_{i,j}g\left(\left[X_i,X_j\right]_{\mm},X^{\ll}\right)g\left(\left[X_i,X_j\right]_{\mm},X^\pp\right) \\
	=&\ \frac{1}{2}\sum_{i,j}g\left(\left[X_i,X_j\right]_{\mm},X^{\ll}\right)g\left(\left[X_i,X_j\right]_{\mm},X^\pp\right).
\end{align*}
We used that $B(\ll, \pp)=0$ in the second equality and in the third equality we used both the Cartan decomposition relations $[\pp,\pp] \subset \kk$, $[\kk, \pp] \subset \pp$, $[\kk,\kk] \subset \kk$ and that $g(\ll,\pp)=0$. Moreover, by the same reason, we have that

\begin{align*}
\sum_{i,j}g([X_i,X_j]_{\mm},X^\ll)g([X_i,X_j]_{\mm},X^\pp) =& \sum_{i,j}g\left(\left[X^\ll_i,X^\ll_j\right]_{\mm},X^\ll\right)g\left(\left[X^\ll_i,X^\ll_j\right]_{\mm},X^\pp\right)\\
&+ \frac{1}{2}\sum_{i,j}g\left(\left[X^{\pp}_i,X^{\pp}_j\right]_{\mm},X^{\ll}\right)g\left(\left[X^{\pp}_i,X^{\pp}_j\right]_{\mm},X^{\pp}\right)\\
&+2\sum_{i,j}g\left(\left[X^{\ll}_i,X^{\pp}_j\right]_{\mm},X^{\ll}\right)g\left(\left[X^{\ll}_i,X^{\pp}_j\right]_{\mm},X^{\pp}\right) \\
=& \ 0.
\end{align*}

This means that the set of awesome metrics $\{g \in \left(\text{Sym}^2(\mm)\right)^{\Ad(H)}_+ \ \vert \ g(\ll,\pp)=0\}$ is an invariant subset for the Ricci fow equation \eqref{eq:ricci flow}.
\end{proof}

\begin{remark*}
	Nikonorov had already argued that $\ric_g(\ll,\pp)=0$ for $g$ awesome, in the particular case of the homogeneous space $SO(n,2)/SO(n)$, $n\geq2$ (see \cite[Example 1]{n}). 
	The isotropy representation of $SO(n,2)/SO(n)$, $n\geq2$, has three summands $\ll_1 \subset \kk$, $\pp_1 \subset \pp$, and $\pp_2\subset \pp$, and moreover $\ll_1$ is not isomorphic to $\pp_1$ or $\pp_2$, thus any $SO(n,2)$-invariant metric is awesome.
	He works this example out in more detail in order to show that $SO(n,2)/SO(n)$ admits $SO(n,2)$-invariant Ricci negative metrics but no Einstein metric. 
\end{remark*}

\section{Algebraic bounds for the Ricci curvature}\label{section:3}
We want to understand the long time behavior of an awesome metric under the homogeneous Ricci flow. In order to do that we want to compare an arbitrary awesome metric $g$ to a highly symmetric background metric.

Let us fix $Q := -B|_{\ll\times\ll} +B|_{\pp\times \pp}$ as a background metric. For a given Ad($H$)-invariant inner product $g$ on $\mm$, by Schur's lemma, we can decompose it on $Q$-orthogonal irreducible $\hh$-modules $\mm = \bigoplus_{i=1}^N \mm_i$ such that

\begin{align*}
	g = x_1\cdot Q |_{\mm_1\times\mm_1} \perp ... \perp x_N\cdot Q|_ {\mm_N\times\mm_N},
\end{align*}
for some positive numbers $x_1,...,x_N \in \mathbb{R}$. Note that this decomposition is not necessarily unique, except in the case where all irreducible modules are pairwise inequivalent. Also by Schur's lemma, in each irreducible summand $\mm_i$ the Ricci tensor is given by $\left.\text{ric}_g\right|_{\mm_i\times\mm_i}=r_i\cdot g|_{\mm_i\times\mm_i}$, for $r_1,...,r_N \in \mathbb{R}$. Observe that, in general, the mixed terms $\text{ric}_g(\mm_i,\mm_j)$ for $i\neq j$ are not zero when $\mm_i$ is equivalent to $\mm_j$ as \text{ad}($\hh$)-modules.

Now let $g$ be an awesome metric. Then there is a Cartan decomposition such that $g(\ll, \pp)=0$, hence we can adapt the above decomposition so that
\begin{equation}\label{eq:g diagonal}
	g = l_1 \cdot Q |_{\ll_1\times \ll_1} \perp ... \perp l_n \cdot Q|_ {\ll_v\times \ll_v} \perp p_1 \cdot Q |_ {\pp_1\times \pp_1} \perp ... \perp p_m \cdot Q|_ {\pp_m\times\pp_{m}},
\end{equation}
where $(\ll_1, ..., \ll_n,\pp_{1}, ..., \pp_{m} )=(\mm_1, ..., \mm_{n+m})$ with 
\[\ll = \oplus_{i=1}^n \ll_i \hspace{0.5cm} \text{and} \hspace{0.5cm} \pp = \oplus_{i=1}^m \pp_i,\]
and $(l_1,...,l_n,p_1,...,p_m)=(x_1,...,x_{n+m})$. 

Let us establish the notation $I_{\ll}:=\{1,...,n\}$ and $I_{\pp}:=\{n+1,...,n+m\}$, and let $d_i$ denote the dimension of $\mm_i$ for all $i \in \{1,...,n+m\}$. To simplify notation we are going to write $\hat{i} \coloneqq n+i$. Finally, let us denote
\begin{equation} \label{eq: ricci L diagonal}
	r^{\ll}_i\cdot g|_{\ll_i\times\ll_i}=\text{ric}_g|_{\ll_i\times\ll_i}, \hspace{1cm} \text{for} \hspace{0.2cm}  i=1,...,n 
\end{equation}
and 
\begin{equation}\label{eq:ricci P diagonal}
	r^{\pp}_{i}\cdot g|_{\pp_i\times\pp_i}=\text{ric}_g|_{\pp_i\times\pp_i}, \hspace{1cm} \text{for} \hspace{0.2cm}  i=1,...,m 
\end{equation}
where $(r^{\ll}_1,...,r^{\ll}_n, r^{\pp}_{1},..., r^{\pp}_{m} )=(r_1,...,r_{n+m})$. 

Let us take the following $Q$-orthonormal basis on $\gg$, $\{E_{\alpha}^0\}$ for
 $1 \leq \alpha \leq n$ on $\hh$, and $\{E_\alpha^i\}$ for $1 \leq \alpha \leq d_i$
 on each $\mm_i$,  $i=1, ..., n+m$. Then we can define the following brackets coefficients 
\[	[ijk]:=\sum_{\alpha, \beta, \gamma} Q\left(\left[E^i_ \alpha, E^j_\beta\right], E^ k_\gamma \right)^2. \]
By the Cartan decomposition we get that $\ad(\kk)$ are skew-symmetric and $\ad(\pp)$ are symmetric \cite[Lemma 13.1.3]{hn}, therefore the coefficient $[ijk]$ is invariant under permutations of the symbols $i, j, k$. 

By Schur's lemma we have that the Casimir operator of the $\hh$ action on the irreducible module $\mm_i$, $C_{\mm_i, \hh}:=-\sum_\alpha \left.\text{ad}\left(E_{\alpha}^0\right)\circ \text{ad}\left(E_{\alpha}^0\right)\right|_{\mm_i}$ is given by
\begin{equation}\label{eq:casimir}
	c_i \cdot \text{Id}_{\mm_i} =-\sum_\alpha\left.\text{ad}(E_{\alpha}^0)\circ \text{ad}(E_{\alpha}^0)\right|_{\mm_i},
\end{equation}
with $c_i \geq 0$.

By \cite[Lemma 1.5]{wz} (also \cite[Lemma 1]{n}), we have that for $i=1, \ldots, n+m$,
\begin{equation}\label {eq:d}
0 \leq \sum _{j,k} [ijk] = d_i(1-2c_i)\leq d_i.
\end{equation}

Indeed, a direct computation yields

\begin{align*}
	\sum _{j,k} [ijk]&=  \sum_{\substack{\alpha, \beta, \gamma \\ 1\leq j,k \leq n+m}} Q\left(\left[E^ i_\alpha, E^j_\beta\right], E^ k_\gamma \right)^2 \\
	&= \sum_{\substack{\alpha, \beta, \gamma \\ 0\leq j,k \leq n+m}} Q\left(\left[E^ i_\alpha, E^j_\beta\right], E^ k_\gamma \right)^2-2\sum_{\alpha, \beta,\gamma} Q\left(\left[E^ i_\alpha,E^0_\beta\right],E^i_\gamma\right)^2 \\ 
	&= \sum_{\substack{\alpha, \beta \\ 0\leq j \leq n+m}} Q\left(\left[ E^ i_\alpha, E^j_\beta \right],\left[ E^ i_\alpha, E^j_\beta\right]\right) -2\sum_{\alpha, \beta} Q\left(\left[E^0_\beta,E^ i_\alpha\right],\left[E^0_\beta,E^ i_\alpha\right]\right)\\	
	&=\sum_{\substack{\alpha, \beta, \gamma \\ 0\leq j \leq n+m}} \left \vert Q\left(E^j_\beta ,\left[E^ i_\alpha,\left[ E^ i_\alpha, E^j_\beta \right]\right]\right)\right \vert +2\sum_{\alpha, \beta} Q\left(E^ i_\alpha,\left[E^0_\beta,\left[E^0_\beta,E^ i_\alpha\right]\right]\right)\\
	&=\sum_\alpha \left\vert B\left(E^i_\alpha,E^i_\alpha\right)\right \vert -2\tr C_{\mm_i, \hh} = d_i -2c_id_i \leq d_i
\end{align*} 

Using the above orthonormal basis,  we have the following formula for the Ricci curvature $r_i$ on $\mm_i$, $i=1,..., n+m$, of an awesome metric on $\mm$ (see \cite[Lemma 2]{n}), 
\begin{equation}\label{eq:nikonorov ricci}
	r_i = \frac{b_i}{2x_i} + \frac{1}{4d_i}\sum_{j,k}[ijk]\left(\frac{x_i}{x_k x_j}-\frac{x_k}{x_i x_j}-\frac{x_j}{x_k x_i}\right),
\end{equation}
where $b_i =1$, if $i \in I_{\ll}$, and $b_i =-1$, if $i \in I_{\pp}$. Let us order $\{x_1,...,x_n\} =\{l_1, ...,l_n\}$ as follows
\[0<l_1 \leq ... \leq l_n,\]
and $ \{x_{n+1},...,x_{n+m}\} =\{p_1, ..., p_m\}$ as
\[\ 0<p_{1} \leq ... \leq p_m.\] 

In \cite{n} there are the following estimates for $r^{\pp}_1$ and $r^{\pp}_m$. First, since $\frac{p_1}{p_j}-\frac{p_j}{p_1}\leq 0$ for $1 \leq j \leq m$,

\begin{align}\label{eq:r1}
r^{\pp}_1 &= -\frac{1}{2p_1} + \frac{1}{4d_{\hat{1}}}\sum_{j,k}[\hat{1}jk]\left(\frac{p_1}{x_k x_j}-\frac{x_k}{p_1 x_j}-\frac{x_j}{x_k p_1}\right) \notag \\
& =-\frac{1}{2p_1} +\frac{1}{2d_{\hat{1}}}\sum_{\substack{\hat{j} \in I_{\pp} \\  k \in I_{\ll}}}[\hat{1}\hat{j}k]\left(\left(\frac{p_1}{p_j}-\frac{p_j}{p_1}\right)\frac{1}{l_k}-\frac{l_k}{p_j p_1}\right) \notag \\ 
&\leq -\frac{1}{2p_1} <0
\end{align}
and since $\frac{p_m}{p_j}-\frac{p_j}{p_m}\geq 0$ for $1 \leq j \leq m$,

\begin{align} \label{eq:rm}
	r^{\pp}_m \stackrel{\phantom{\eqref{eq:d}}}{=}& -\frac{1}{2p_m} +\frac{1}{2d_{\hat{m}}}\sum_{\substack{\hat{j} \in I_{\pp} \\ k \in I_{\ll}}}[\hat{m}\hat{j}k]\left(\left(\frac{p_m}{p_j}-\frac{p_j}{p_m}\right)\frac{1}{l_k}-\frac{l_k}{p_j p_m}\right) \notag \\
	\stackrel{\phantom{\eqref{eq:d}}}{\geq}&-\frac{1}{2p_m} -\frac{1}{2d_{\hat{m}}}\sum_{\substack{\hat{j} \in I_{\pp} \\ k \in I_{\ll}}}[\hat{m}\hat{j}k]\frac{l_k}{p_j p_m} \notag \\
	\stackrel{\eqref{eq:d}}{\geq}& -\frac{1}{2p_m}-\frac{l_n}{4p_1p_m}. 
\end{align} 

We now observe that for an awesome metric $g$ the Ricci tensor restricted to the tangent space $T_pK/H$, which can be identified with $\ll$ via \eqref{eq:awesome cartan decomposition}, splits nicely in terms of $\ll$ and $\pp$. Namely, for any $X \in \ll$, the Ricci tensor formula \eqref{eq:ricci formula} for an awesome metric gives us

\begin{align}\label{eq:ricci tensor of in terms of the induced metric}
	\ric_g(X,X)&=\text{ric}_{K/H}\left(X,X\right)-\frac{1}{2}\tr\left(\left.\text{ad}(X)\circ \text{ad}(X)\right|_{\pp}\right) \\
	&-\frac{1}{2}\sum_i \left|[X,X^{\pp}_i]_{\mm}\right|_g^2  + \frac{1}{4}\sum_{i,j}g\left(\left[X^{\pp}_i,X^{\pp}_j\right]_{\mm},X\right)^2,\notag
\end{align}
where $\{X^{\pp}_i\}$ is an $g$-orhtonormal basis for $\pp$ and $\text{ric}_{K/H}$ is the Ricci tensor on $K/H= K \cdot p$. In particular, we have the following lemma.

\begin{lemma} \label{lemma:compact ricci curv bound}
Let $(G/H,g)$ be a semisimple homogeneous space with an awesome $G$-invariant metric $g$. Then the Ricci curvature $r^{\ll}_n$ in the largest $\ll$-eigendirection of $g$ with respect to the background metric $Q$ satisfies
\begin{equation} \label{eq:rn}
	r^{\ll}_n \geq \frac{1}{4d_n}\sum_{\hat{j},\hat{k} \in I_{\pp}}[n\hat{j}\hat{k}]\left(\frac{2}{l_n} + \frac{l_n}{p_j p_k}-\frac{p_j}{l_n p_k}-\frac{p_k}{p_j l_n}\right)+ \frac{1}{4d_nl_n}\sum_{j,k \in I_{\ll}}[njk].
\end{equation} 
\end{lemma}
\begin{proof}
Observe that as computed in \cite[Theorem 1]{n} and in \cite[Theorem 3.1]{boe} for $ i>j \in I_\ll$

\begin{align*}
 l_j^2-2l_il_j+l_i^2 = \left(l_j - l_i\right)^2 &\leq \left(l_n-l_j\right)\left(l_i-l_j\right) \\
&=  l_n l_i -l_i l_j - l_nl_j+l_j^2 \\
&\leq l_n^2-l_i l_j-l_n l_j + l_j l_n \\
&=l_n^2-l_i l_j
\end{align*}
hence $l_n^2 -l_j^2-l_i^2+l_jl_i  \geq 0 $
and the formula for $r_n$ with the Cartan decomposition relations then yields

	\begin{align*}
	r^{\ll}_n \stackrel{\phantom{\eqref{eq:d}}}{=}& \frac{1}{2l_n} + \frac{1}{4d_n}\sum_{j,k}[njk]\left(\frac{l_n}{x_k x_j}-\frac{x_j}{l_n x_j}-\frac{x_j}{x_k l_n}\right) \\
	\stackrel{\phantom{\eqref{eq:d}}}{=}&\frac{d_n}{2l_nd_n} + \frac{1}{4d_n}\sum_{\hat{j},\hat{k} \in I_{\pp}}[n\hat{j}\hat{k}]\left(\frac{l_n}{p_k p_j}-\frac{p_i}{l_n p_j}-\frac{p_j}{p_k l_n}\right)+ \frac{1}{4d_n}\sum_{j,k \in I_{\ll}}[njk]\left(\frac{l_n}{l_k l_j}-\frac{l_k}{l_n l_j}-\frac{l_j}{l_k l_n}\right) \\
	\stackrel{\eqref{eq:d}}{\geq}& \frac{1}{4d_n}\sum_{\hat{j},\hat{k} \in I_{\pp}}[n\hat{j}\hat{k}]\left(\frac{2}{l_n} + \frac{l_n}{p_k p_j}-\frac{p_k}{l_n p_j}-\frac{p_j}{p_k l_n}\right)+ \frac{1}{4d_n}\sum_{j,k \in I_{\ll}}[njk]\left(\frac{2}{l_n} + \frac{l_n}{l_k l_j}-\frac{l_k}{l_n l_j}-\frac{l_j}{l_k l_n}\right) \\
	\stackrel{\phantom{\eqref{eq:d}}}{=}& \frac{1}{4d_n}\sum_{\hat{j},\hat{k} \in I_{\pp}}[n\hat{j}\hat{k}]\left(\frac{2}{l_n} + \frac{l_n}{p_kp_j}-\frac{p_k}{l_n p_j}-\frac{p_j}{p_k l_n}\right)+ \frac{1}{4d_nl_n}\sum_{j,k \in I_{\ll}}[njk]\left(1+\frac{l_n^2 -l_j^2-l_k^2+l_jl_k}{l_kl_j}\right) \\
	\stackrel{\phantom{\eqref{eq:d}}}{\geq}& \frac{1}{4d_n}\sum_{\hat{j},\hat{k} \in I_{\pp}}[n\hat{j}\hat{k}]\left(\frac{2}{l_n} + \frac{l_n}{p_k p_j}-\frac{p_k}{l_n p_j}-\frac{p_j}{p_k l_n}\right)+ \frac{1}{4d_nl_n}\sum_{j,k \in I_{\ll}}[njk].
\end{align*}
\end{proof}

\begin{remark*}\label{remark:toral case}
Observe that in general $\sum_{j,k \in I_{\ll}}[njk]$ may be equal to zero. But in the case when $K/H$ is not a torus, $[\ll,\ll]_{\mm} \neq0$ and we can consider the largest eigenvalue $l_{n'}$ such that $[\ll_{n'},\ll]_{\mm} \neq 0$, for which $\sum_{j,k \in I_{\ll}}[n'jk]>0$. 
\end{remark*}

We have also the following alternative for the values of $r^{\ll}_n$ and $r^{\pp}_m$. 
\begin{lemma} \label{lemma:Ricci dichotomy}
Let $(G/H,g)$ be a semisimple homogeneous space with an awesome $G$-invariant metric $g$. Then we have the following dichotomy for the Ricci curvature of $g$ in the largest $\pp$(respectively $\ll$)-eigendirection of $g$ with respect to $Q$.
\begin{enumerate}
\item If $p_m - p_1 \geq l_n$, then
 \begin{equation} \label{eq:option1}
	\hspace{1.8cm} r^{\pp}_m \geq -\frac{1}{4p_m} - \frac{1}{4p_1} \hspace{0.5cm} \text{and} \hspace{0.5cm}
	r^{\ll}_n \geq \frac{1}{4l_n}\left(2 -\frac{p_m}{p_1}-\frac{p_1}{p_m}\right).
\end{equation} 
\item If $p_m - p_1 \leq l_n$, then 
\begin{equation}\label{eq:option2}
	\hspace{-1.5cm}r^{\pp}_m \geq -\frac{1}{2p_m}-\frac{l_n}{4p_1p_m} \hspace{0.5cm} \text{and} \hspace{0.5cm}
	r^{\ll}_n \geq 0.
\end{equation}
\end{enumerate}
\end{lemma}

\begin{proof}
		In the first case, i.e. if $p_m - p_1 \geq l_n $, then 
		
		\begin{align*}
		r^{\pp}_m \stackrel{\eqref{eq:rm}}{\geq}& -\frac{1}{2p_m}-\frac{l_n}{4p_1p_m} \\
		\stackrel{\phantom{\eqref{eq:rm}}}{\geq}& -\frac{1}{2p_m}+\frac{p_1-p_m}{4p_1p_m} \\
		\stackrel{\phantom{\eqref{eq:rm}}}{=}&-\frac{1}{4p_m}-\frac{1}{4p_1}.
	\end{align*}
	Moreover, by Lemma \ref{lemma:compact ricci curv bound} we have that
	
	\begin{align*}
		r^{\ll}_n &\geq \frac{1}{4d_n}\sum_{\hat{j},\hat{k} \in I_\pp}[n\hat{j}\hat{k}]\left(\frac{2}{l_n} + \frac{l_n}{p_k p_j}-\frac{p_k}{l_n p_j}-\frac{p_j}{p_k l_n}\right)+ \frac{1}{4d_nl_n}\sum_{j,k \in \ll}[njk]\\
		&\geq \frac{1}{4l_n}\left(2 -\frac{p_m}{p_1}-\frac{p_1}{p_m}\right),
	\end{align*}
	and we get the first inequalities.
	
In the second case, as in \cite{n}, observe that

	\begin{align*}
		p_m - p_1 \leq l_n  \iff& |p_j - p_k| \leq l_n, \forall j,k \in I_\pp \\
		\iff& \left(p_j^2 - 2p_jp_k + p_k^2 \right) \leq l_n^2 \\ 
		\iff& \left(\frac{p_j}{p_k l_n}-\frac{2}{l_n} + \frac{p_k}{p_jl_n}\right) \leq \frac{l_n}{p_j p_k} \\
		\iff& -\frac{2}{l_n} \leq  \frac{l_n}{p_j p_k} -\frac{p_j}{p_k l_n} - \frac{p_k}{p_jl_n}
	\end{align*}
	and therefore, substituting this in the estimate for $r^\ll_n$ in \eqref{eq:rn} we get that $r^\ll_n \geq 0$. Together with the estimate \eqref{eq:rm} for $r^\pp_m$, we get the desired inequalities.
\end{proof}

We can combine the estimates above to get a first scale-invariant estimate that will be essential for us in the dynamical analysis to come in Section \ref{section:4}

\begin{lemma} \label{lemma:Ricci final estimate}
Let $(G/H,g)$ be a semisimple homogeneous space with an awesome $G$-invariant metric $g$. Let us consider the eigenvalues of $g$ with respect to $Q$ as established in \eqref{eq:g diagonal}, \eqref{eq: ricci L diagonal}, and \eqref{eq:ricci P diagonal}. Then

\begin{equation}\label{eq:geometric inequality1}
2(p_mr^{\pp}_m+l_nr^{\ll}_n) \geq -\frac{p_m+l_n}{p_1}.
\end{equation}
\end{lemma}
\begin{proof}
We know that if $p_m-p_1 \geq l_n$, then 

\begin{align*}
	2\left(p_mr^{\pp}_m+l_nr^{\ll}_n\right) \stackrel{\eqref{eq:option1}}{\geq}& -\frac{1}{2} -\frac{p_m}{2p_1} +\frac{1}{2}\left(2 -\frac{p_m}{p_1}-\frac{p_1}{p_m}\right) \\
	\stackrel{\phantom{\eqref{eq:option1}}}{=}& \frac{1}{2} -\frac{p_m}{p_1}-\frac{p_1}{2p_m} = -\frac{p_m}{p_1} + \frac{p_m-p_1}{2p_m}\\
	\stackrel{\phantom{\eqref{eq:option1}}}{\geq}& -\frac{p_m}{p_1} \geq -\frac{p_m + l_n}{p_1}.
\end{align*}
And if $p_m-p_1 \leq l_n$, then

\begin{align*}
2\left(p_mr^{\pp}_m+l_nr^{\ll}_n\right) &\stackrel{\eqref{eq:option2}}{\geq} -1-\frac{l_n}{2p_1} \geq -\frac{2p_1+2l_n}{2p_1} \geq -\frac{p_m + l_n}{p_1}.
\end{align*}
\end{proof}

We can observe in the above computations that the advantage of working with awesome metrics is that the Ricci curvature in the eigenvectors of the metric tangent to $K/H \subset G/H$ splits nicely, since the algebraic and the metric decompositions are compatible. Lemmas \ref{lemma:Ricci dichotomy} and \ref{lemma:Ricci final estimate} give us our first estimates on the Ricci curvature which we will be able to exploit in the next section in order to examine the long-time behavior of the Ricci flow on awesome metrics. 

\section{The Dynamical Alekseevskii Conjecture for awesome metrics}\label{section:4}
In this section we will prove Theorem \ref{thm:A}. In order to do that, we must use the algebraic estimates we got on the last section in order to get dynamical estimates for the eigenvalues of the Ricci flow solution $g(t)$ (which from here onward we will also denote by $g_t$) with respect to our background metric $Q\coloneqq-B|_{\ll\times\ll} +B|_{\pp\times\pp}$.

The next lemma is equivalent to the one in \cite[Lemma B.40]{chow}, but here we use for better convenience upper left-hand Dini derivatives.

\begin{lemma} [Lemma B.40, \cite{chow}] \label{lemma:Dini}
	Let $C$ be a compact metric space, $I$ an interval of $\mathbb{R}$, and a function $g: C \times I \to \mathbb{R}$, such that $g$ and $\frac{\partial g}{\partial t}$ are continuous. Define $\phi \colon I \to \mathbb{R}$ by
	\[\phi(t) \coloneqq \sup_{x \in C}g(t,x)\]
	and its upper left-hand Dini derivative by
	\[\frac{d^-\phi(t) }{dt} \coloneqq \limsup_{h \to 0^+}\frac{\phi(t)-\phi(t-h)}{h}.\]
	
	Let $C_t \coloneqq \{x \in C \ | \ \phi(t)=g(t,x)\}$. We have that $\phi$ is continuous and that for any $t \in I$
	\[\frac{d^-\phi(t) }{dt} = \min_{x \in C_t}\frac{\partial g}{\partial t}(t,x).\]
\end{lemma} 
At first, we will apply the lemma above to
$-p_1(t) = \max\{-g_t(X,X) \ | \ X \in \pp, \Vert X\Vert_Q=1\}$, $p_m(t) = \max\{g_t(X,X) \ | \ X \in \pp, \Vert X\Vert_Q=1\}$, and $l_n(t) = \max\{g_t(X,X) \ | \ X \in \ll, \Vert X\Vert_Q=1\}$. As well as to 

\begin{align*}
	\log(p_m +l_n)(t)&=\log(\max\{(g_t(X,X)+ g_t(Y,Y)) \ | \ X \in \pp, \ Y \in \ll, \ \Vert X\Vert_Q=1, \ \Vert Y\Vert _Q=1\})\\
	&=\max\{\log(g_t(X,X)+ g_t(Y,Y)) \ | \ X \in \pp, \ Y \in \ll, \ \Vert X\Vert _Q=1, \ \Vert Y\Vert _Q=1\}. 
\end{align*}

This lemma will be fundamental to us when combined with the following elementary real analysis result. 
\begin{lemma} \label{lemma:Dini Dynamic}
Let $[a,b]$ be a closed interval of $\mathbb{R}$ and $f: [a,b] \to \mathbb{R}$ a continuous function such that $\frac{d^-f}{dt} \leq 0$ and $f(a)=0$. Then $f(t) \leq 0$. 
\end{lemma} 

With Lemmas \ref{lemma:Dini} and \ref{lemma:Dini Dynamic} at hand, we can obtain our first main estimate for analyzing the long-time behavior of the Ricci flow on an awesome metric. Indeed, using the algebraic estimates in the previous section we get that 

\begin{lemma}\label{lemma:at most linear growth}
	Let $(G/H,g(t))$, $t \in [0,T)$, be an awesome Ricci flow adapted to the Cartan decomposition $\gg = \kk \oplus \pp$. Then we have the following upper bound for the growth of $p_m(t)$ and $l_n(t)$,
	\begin{equation}
		t +p_1(0)\leq p_1(t)
	\end{equation}
	and
	\begin{equation}
		(p_m+l_n)(t) \leq (t+p_1(0))\frac{(p_m+l_n)(0)}{p_1(0)}.
	\end{equation}
\end{lemma}
\begin{proof}
 Let us get first the estimate for $p_1(t)$. Using Lemma \ref{lemma:Dini} and the estimate \eqref{eq:r1} we get that

\begin{align*}
\frac{d^-(-p_1)}{dt} \leq 2\ric_t(E^{\hat{1}},E^{\hat{1}})=2r^\pp_1 p_1 \stackrel{\eqref{eq:r1}}{\leq} -1,
\end{align*}
hence by Lemma \ref{lemma:Dini Dynamic}
\begin{equation}\label{eq:p1}
	t + p_1(0) \leq p_1(t)
\end{equation}

Moreover, using Lemma \ref{lemma:Ricci dichotomy} we get the following estimate

	\begin{align*}
	\frac{d^-\log(p_m +l_n)}{dt} \ \stackrel{\phantom{(}\ref{lemma:Dini}\phantom{)}}{\leq} & \  \frac{d}{dt}\log\left(g_t(E^{\hat{m}},E^{\hat{m}})+g_t(E^n,E^n)\right)= \frac{g'_t(E^{\hat{m}},E^{\hat{m}})+g'_t(E^n,E^n)}{g_t(E^{\hat{m}},E^{\hat{m}})+g_t(E^n,E^n)} \\
	\stackrel{\phantom{\eqref{eq:p1}}}{=}& \frac{-2r^\pp_mp_m  -2r^\ll_nl_n}{p_m+l_n} \stackrel{\eqref{eq:geometric inequality1}}{\leq}\frac{1}{p_1} \\
	\stackrel{\eqref{eq:p1}}{\leq} & \frac{1}{t+p_1(0)},
\end{align*}
hence $(p_m+l_n)(t) \leq (t+p_1(0))\frac{(p_m+l_n)(0)}{p_1(0)}$.
\end{proof}

We see then that the maximum eigenvalue of $g(t)$ with respect to $Q$ can grow at most as $O(t)$. In particular, $p_m(t) \leq (1+c_0)(t+p_1(0))$ for some given positive constant $c_0>0$ that only depends on the initial metric $g_0$.

Therefore, using \eqref{eq:p1} we get a pinching estimate along the flow for the metric $\left.g(t)\right|_{\pp\times \pp}$, namely
\begin{equation}\label{eq:pm/p1}
\frac{p_m(t)}{p_1(t)} \leq 1+c_0.
\end{equation}
We can then re-use this estimate to get the following proposition (for the sake of simplicity, from now on we will mostly omit that we are using the Lemmas \ref{lemma:Dini} and \ref{lemma:Dini Dynamic} above).

\begin{proposition}\label{prop:pinching}
	Let $(G/H,g(t))$, $t \in [0,T)$, be an awesome Ricci flow adapted to the Cartan decomposition $\gg = \kk \oplus \pp$. Let us consider the eigenvalues $p_1(t)$, $p_m(t)$, and $l_n(t)$ of $g_t$ with respect to $Q$. Then for all $t \geq t_0$ in the maximal interval of the Ricci Flow, there is an explicit constant $c_0>0$, which only depends on the initial conditions at $t_0$, such that 
	\begin{equation}\label{eq:p_m(t) at most p_1(t)+sublinear}
		p_m(t) \leq t+\left(p_1(t_0)-t_0\right) + c_0\sqrt{t+\left(p_1(t_0)-t_0\right)}\leq p_1(t) + c_0 \sqrt{p_1(t)}
	\end{equation}
and 
	\begin{equation}\label{eq:l(t) sublinear}
	l_n(t) \leq c_0\sqrt{(t-t_0+ p_1(t_0))}.
	\end{equation}
\end{proposition}
\begin{proof}
By rescaling the initial metric $g(t_0)$ we may assume without loss of generality that $p_1(t_0) = 1$. Let $t \geq t_0$. We will show by induction that for all $N \in \mathbb{N}$,

	\begin{align*}
		p_m(t) &\leq \left(1+\frac{c_0}{2^N}\right)\left(t-t_0 +1\right)+c_0 \sum_{k=0}^{N-1}\frac{(\log \left(t-t_0 +1\right))^k}{2^k k!} 
	\end{align*}
	and
	
	\begin{align*}
		l_n(t) &\leq \frac{c_0}{2^N}\left(t-t_0 +1\right)+c_0 \sum_{k=0}^{N-1}\frac{(\log \left(t-t_0 +1\right))^k}{2^k k!},
	\end{align*}
	where $c_0 \coloneqq p_m(t_0) +l_n(t_0)-1$.
	
		By Lemma \ref{lemma:at most linear growth}, we have already seen that $p_m(t) \leq (1+c_0)\left(t-t_0 +1\right)$. Moreover, by the discussion following Lemma \ref{lemma:Ricci dichotomy} we get
		
	\begin{align*}
	\frac{d^-l_n(t)}{dt} \leq -2r_n(t)l_n(t) \leq \frac{1}{2}\left(\frac{p_m(t)}{p_1(t)}+\frac{p_1(t)}{p_m(t)} -2\right) \leq \frac{1}{2}\left(\frac{p_m(t)}{p_1(t)}-1\right) \stackrel{\eqref{eq:pm/p1}}{\leq} \frac{c_0}{2}.
\end{align*}
Hence, for $t \geq t_0$,

	\begin{equation} \label{eq:ln}
		l_n(t) \leq \frac{c_0}{2}\left(t-t_0+1\right)-\frac{c_0}{2}+l_n(t_0) \leq\frac{c_0}{2}\left(t-t_0+1\right)+c_0.
	\end{equation}

	Moreover, 
	
	\begin{align*}
		\frac{d^-p_m(t)}{dt} \leq -2r_m(t)p_m(t) \stackrel{\eqref{eq:rm}}{\leq} 1+\frac{l_n(t)}{2p_1(t)} \stackrel{\eqref{eq:ln}}{\leq}& \left(1+\frac{c_0}{2^2}\right)\frac{(t-t_0+1)}{p_1(t)}+\frac{c_0}{2p_1(t)}\\
		\stackrel{\phantom{\eqref{eq:ln}}}{\leq}& \left(1+\frac{c_0}{2^2}\right)+\frac{c_0}{2(t-t_0+1)}.
	\end{align*}
	Hence, for $t\geq t_0$,
	
	\begin{align*}
		p_m(t) \leq \ \left(1+\frac{c_0}{2^2}\right)(t-t_0+1) + \frac{c_0}{2}\log (t-t_0+1) + c_0.
	\end{align*}
	
	This establishes the basis of induction. Suppose that for $N \geq 2$ and $t \geq t_0$,
	
	\begin{align*}
		p_m(t) &\leq \left(1+\frac{c_0}{2^N}\right)(t-t_0+1)+c_0 \sum_{k=0}^{N-1}\frac{(\log (t-t_0+1))^k}{2^k k!} 
	\end{align*}
	and
	
	\begin{align*}
		l_n(t) &\leq \frac{c_0}{2^N}(t-t_0+1)+c_0 \sum_{k=0}^{N-1}\frac{(\log (t-t_0+1))^k}{2^k k!}.
	\end{align*}
	
	Then we can re-use this to get better estimates for $l_n(t)$ and $p_m(t)$. We have that
	
	\begin{align*}
		\frac{d^-l_n(t)}{dt} \leq \frac{1}{2}\left(\frac{p_m(t)}{p_1(t)}-1\right) \leq \frac{1}{2}\left(\frac{p_m(t)}{(t-t_0+1)}-1\right) \leq \frac{c_0}{2^{N+1}}+c_0 \sum_{k=0}^{N-1}\frac{(\log (t-t_0+1))^k}{(t-t_0+1)2^{k+1} k!} 
	\end{align*}
	hence
	
	\begin{align*}
		l_n(t) \leq \frac{c_0}{2^{N+1}}(t-t_0+1)+c_0 \sum_{k=0}^{N}\frac{(\log (t-t_0+1))^k}{2^{k} k!}, 
	\end{align*}
	and
	
	\begin{align*}
		\frac{d^-p_m(t)}{dt} \leq 1+\frac{l_n(t)}{2p_1(t)} \leq \left(1+\frac{c_0}{2^{N+1}}\right)+c_0 \sum_{k=0}^{N-1}\frac{(\log (t-t_0+1))^k}{(t-t_0+1)2^{k+1} k!} .
	\end{align*}
	Therefore,
	
	\begin{align*}
		p_m(t) \leq \ \left(1+\frac{c_0}{2^{N+1}}\right)(t-t_0+1) +c_0 \sum_{k=0}^{N}\frac{(\log (t-t_0+1))^k}{2^{k} k!}.
	\end{align*}
	
	Now this is valid for arbitrary $t \geq t_0$ in the maximal interval of the dynamics. Hence, taking the limit at $N \to \infty$ for the right-hand side, we get
	
	\begin{equation*}
		p_m(t) \leq (t-t_0+1) +c_0 \sqrt{(t-t_0+1)} \leq p_1(t) + c_0\sqrt{p_1(t)}
	\end{equation*}
and

\begin{equation*}
	l_n(t) \leq c_0 \sqrt{(t-t_0+p_1(t_0))}.
\end{equation*}
\end{proof}

The next corollary follows immediately from the estimate \eqref{eq:p_m(t) at most p_1(t)+sublinear} obtained in Proposition \ref{prop:pinching} above.

\begin{corollary}\label{corollary:p/p goes to 1}
Let $(G/H,g(t))$, $t \in [t_0,\infty)$, be an immortal awesome Ricci flow adapted to the Cartan decomposition $\gg = \kk \oplus \pp$. Then 
\[\lim_{t\to\infty}\frac{p_m(t)}{p_1(t)}=1.\]
\end{corollary}

\begin{remark*}
Observe that the background metric $Q$ given by $-B|_{\kk\times\kk} +B|_{\pp\times\pp}$ corresponds to a minimum for the moment map of the $SL(\gg,\mathbb{R})$-action of determinant-one matrices in the space of Lie brackets on $\gg$ (\cite[Proposition 8.1]{lau}). Thus, Corollary \ref{corollary:p/p goes to 1} reinforces the relation between geometric invariant theory and the geometry of $G/H$ by telling us that in the non-compact part $\pp$ the awesome metric under the Ricci flow approximates $-B|_{\pp\times\pp}$ up to scaling and $G$-equivariant isometry. In Section \ref{section:5} we will prove a convergence result in this direction (cf. Theorem \ref{theo:conv immortal}).
\end{remark*}

The next corollary of Proposition \ref{prop:pinching}, about the asymptotic behavior of the metric $g(t)$ restricted to the tangent space of $K/H$, will also be important for the long-time analysis of immortal awesome homogeneous Ricci flows in Section \ref{section:5}.

\begin{corollary}\label{corollary:g_s goes to 0}
Let $(G/H,g(t))$, $t \in [t_0,\infty)$, be an immortal awesome Ricci flow adapted to the Cartan decomposition $\gg = \kk \oplus \pp$. Then the rescaled metric $\tilde{g}_t \coloneqq t^{-1}g(t)$ is such that 
	\[\lim_{t\to \infty}\left.\tilde{g}_t\right|_{\ll\times \ll} =0.\]
\end{corollary}
\begin{proof}
	This follows immediately from the estimate \eqref{eq:l(t) sublinear} for $l_n(t)$, which shows that it may grow at most sublinearly. Hence, $\lim_{t\to\infty}\frac{l_n(t)}{t}=0$.\\
\end{proof}
We can now prove the first main result of this article.
\begin{theorem} \label{theo:main}
Let $(M ^d=G/H,g_0)$ be a semisimple homogeneous space such that the universal cover is not diffeomorphic to $\mathbb{R}^d$. If $g_0$ is an awesome metric, then the Ricci flow solution starting at $g_0$ has finite extinction time.
\end{theorem}

\begin{proof}
Let $\gg$ and $\hh$ be the Lie algebras of $G$ and $H$ respectively. Let us consider the Cartan decomposition $\gg = \kk \oplus \pp$ with $\hh \subset \kk$ and let us fix as background metric $Q \coloneqq -B|_{\ll\times \ll} +B|_{\pp\times \pp}$, where $B$ is the Killing form in $\gg$. We have then the canonical reductive identification $T_{eH}G/H \cong \mm = \ll \oplus \pp$, with $\ll := \hh^{\perp_Q}\cap\kk$, where $\hh$ is the Lie algebra of $H$. First, observe that as by hypothesis the universal cover of $M$ is not diffeomorphic to $\mathbb{R}^d$, then $K/H$ is not a torus, hence $[\kk,\kk] \not\subset \hh $. Moreover, since $[\kk,\kk] \perp_Q \zz(\kk)$, where $\zz(\kk)$ is the center of $\kk$, the condition $[\kk,\kk] \subset \hh$ is equivalent to $\ll \subset \zz(\kk)$, which in turn is equivalent to $[\ll,\ll] \subset \hh$ and $[\hh, \ll]=0$. 

So in terms of the irreducible representations decomposition \eqref{eq:g diagonal}, $[\kk,\kk] \not\subset \hh$ is equivalent to say that for at least one $i \in I_{\ll}$ either: $[\ll,\ll] \not\subset \hh$ and
\[\sum_{j,k \in I_{\ll}}[i jk] >0;\]
or $[\hh,\ll] \neq0$ and
\[\sum_{j,k}[i jk] \stackrel{\eqref{eq:d}}{=}d_i(1-2c_i)<d_i,\]
since then the Casimir operator $C_{\ll_i, \hh}$ given in $\eqref{eq:casimir}$ is not zero. 

Therefore, there is a constant $\lambda >0$ such that for any of the irreducible $\hh$-modules $\ll_i$ with $[\ll_i,\kk] \not\subset \hh$, either 
\[\sum_{j,k \in I_{\ll}}[i jk] \geq 2\lambda >0\]
or
\[\sum_{j,k}[i jk] \leq d_i-\lambda<d_i.\]

Now let $g_t$, $t \in [0,T)$, be the Ricci Flow \eqref{eq:ricci flow} solution starting at $g_0$. By Lemma \ref{lemma:awesome is RF invariant} we know that $g_t$ is an awesome homogeneous Ricci flow adapted to the Cartan decomposition $\gg = \kk \oplus \pp$, with $\kk = \hh\oplus\ll$. 

Let us consider the diagonalization of $g(t)$ with regard to $Q$ as in \eqref{eq:g diagonal},
\[	g_t = l_1(t)\cdot Q |_{\ll_1\times \ll_1} \perp \ldots \perp l_n(t)\cdot Q|_ {\ll_n\times \ll_n} \perp p_1(t)\cdot Q |_ {\pp_1\times \pp_1} \perp \ldots \perp p_m(t)\cdot Q|_ {\pp_m\times \pp_m}.\]
Let us denote the Ricci curvatures, as in \eqref{eq: ricci L diagonal} and \eqref{eq:ricci P diagonal}, by

\[r^{\ll}_i\cdot g(t)|_{\ll_i\times \ll_i}=\ric_{g(t)}|_{\ll_i\times \ll_i}, \hspace{1cm} \text{for} \hspace{0.2cm}  i=1,...,n \]
and 
\[r^{\pp}_i\cdot g(t)|_{\pp_i\times \pp_i}=\ric_{g(t)}|_{\pp_i\times \pp_i}, \hspace{1cm} \text{for} \hspace{0.2cm}  i=1,...,m. \]

Define the $\hh$-submodules $V = \{X \in W^{\perp_Q} \ | \ [\ll,X] \in \hh \}$ and  $W = \{X \in \ll \ | \ [\hh,X] =0 \}$, and let 
\[L(t) := \max\{g_t(X,X) \ | \ X \in V^{\perp_Q} \cup W^{\perp_Q} \ , \ \Vert X\Vert_Q=1 \}.\]
Note that $L(t)$ is the largest eigenvalue $l_{n'}(t)$ of $g(t)$ such that corresponding irreducible $\hh$-module $\ll_{n'}$, satisfies either $[\ll_n, \ll] \not \subset \hh$ or $[\hh,\ll_n] \neq 0$, which we already know it is a non-empty condition. Let us moreover consider $p_1(t) := \min\{g_t(X,X)  | X \in \pp, \Vert X\Vert_Q=1\}$ and $p_m(t) := \max\{g_t(X,X) \ | \ X \in \pp, \Vert X\Vert_Q=1\}$. Then using all the estimates we got so far and the fact that for any $j \in I_{\ll}$, if $\sum_{k \in I_\ll}[n'jk]\neq 0$, then $l_{n'} \geq l_j$, we get the following lower bound for $r^{\ll}_{n'}$,

	\begin{align*}
	r^{\ll}_{n'} \stackrel{\eqref{eq:nikonorov ricci}}{=}&  \frac{1}{2l_{n'}} + \frac{1}{4d_{n'}}\sum_{j,k}[n'jk]\left(\frac{l_{n'}}{x_k x_j}-\frac{x_j}{l_{n'} x_j}-\frac{x_j}{x_k l_{n'}}\right) \\
	\stackrel{\phantom{\eqref{eq:nikonorov ricci}}}{=}& \frac{d_{n'}}{2l_{n'}d_{n'}} + \frac{1}{4d_{n'}}\sum_{\hat{j},\hat{k} \in I_\pp}[n'\hat{j}\hat{k}]\left(\frac{l_{n'}}{p_k p_j}-\frac{p_i}{l_{n'} p_j}-\frac{p_j}{p_k l_{n'}}\right)\\
	&+\frac{1}{4d_{n'}} \sum_{j,k \in I_\ll}[n'jk]\left(\frac{l_{n'}}{l_k l_j}-\frac{l_k}{l_{n'} l_j}-\frac{l_j}{l_k l_{n'}}\right)\ \\
	\stackrel{\phantom{\eqref{eq:nikonorov ricci}}}{\geq}& \frac{d_{n'}}{2l_{n'}d_{n'}} + \frac{1}{4d_{n'}}\sum_{\hat{j},\hat{k} \in I_\pp}[n'\hat{j}\hat{k}]\left(-\frac{p_i}{l_{n'}p_j}-\frac{p_j}{p_k l_{n'}}\right)+ \frac{1}{4d_{n'}l_{n'}}\sum_{j,k \in I_\ll}[n'jk]\left(\frac{l_{n'}^2 -l_j^2-l_k^2}{l_kl_j}\right) \\
	\stackrel{\phantom{\eqref{eq:nikonorov ricci}}}{\geq}& \frac{d_{n'}}{2l_{n'}d_{n'}} -\frac{1}{2d_{n'}l_{n'}}\sum_{\hat{j},\hat{k} \in I_\pp}[n'\hat{j}\hat{k}]\frac{p_m}{ p_1}- \frac{1}{4d_{n'}l_{n'}}\sum_{j,k \in I_\ll}[n'jk].
\end{align*}
Therefore,

\begin{align*}
\frac{d^-l_{n'}(t)}{dt} \leq -2r^{\ll}_{n'}(t) l_{n'}(t)\leq \frac{1}{d_{n'}}\left(- d_{n'} +\frac{1}{2}\sum_{j,k \in I_\ll}[n'jk] + \sum_{\hat{j},\hat{k} \in I_\pp}[n'\hat{j}\hat{k}]\frac{p_m(t)}{ p_1(t)}\right).
\end{align*}

Let us assume by contradiction that $g(t)$ is immortal. Then by Corollary \ref{corollary:p/p goes to 1}, given $\epsilon >0$ we can assume that $t$ is big enough such that 
\[\frac{p_m(t)}{p_1(t)} \leq 1+\epsilon,\]
therefore we get that

\begin{align*}
\frac{d^-l_{n'}(t)}{dt} &\leq \frac{1}{d_{n'}}\left(- d_{n'} +\frac{1}{2}\sum_{j,k \in I_\ll}[n'jk] + \sum_{\hat{j},\hat{k} \in I_\pp}[n'\hat{j}\hat{k}] + \epsilon \sum_{\hat{j},\hat{k} \in I_\pp}[n'\hat{j}\hat{k}]\right) \\
&= \frac{1}{d_{n'}}\left((1+\epsilon)\sum_{j,k}[n'jk]- d_{n'} -\left(\frac{1}{2}+\epsilon\right)\sum_{j,k \in I_\ll}[n'jk]\right)
\end{align*}
and since either $[\hh,\ll_{n'}] \neq 0$  or $[\ll_n, \ll] \not \subset \hh$, there exists a constant independent of $t$, $\lambda >0$, such that either $\sum_{j,k}[n'jk] =d_{n'} - \lambda < d_{n'}$ or $\sum_{j,k \in I_\ll}[n'jk] \geq 2\lambda>0$. Therefore, if $d=\dim M$, then either 

\begin{align*}
\frac{d^-l_{n'}(t)}{dt}&= \frac{1}{d_{n'}}\left((1+\epsilon)\sum_{j,k}[n'jk]- d_{n'} -\left(\frac{1}{2}+\epsilon\right)\sum_{j,k \in I_\ll}[n'jk]\right) \\ 
& = \frac{1}{d_{n'}}\left((1+\epsilon)(d_{n'} -\lambda)-d_{n'}\right) \leq -\frac{\lambda}{d}+\epsilon
\end{align*}
or

\begin{align*}
\frac{d^-l_{n'}(t)}{dt} &= \frac{1}{d_{n'}}((1+\epsilon)\sum_{j,k}[n'jk]- d_{n'} -\left(\frac{1}{2}+\epsilon\right)\sum_{j,k \in I_\ll}[n'jk]) \\ 
& \leq \frac{1}{d_{n'}}((1+\epsilon)(d_{n'})-d_{n'}) -\frac{\lambda}{d}\leq -\frac{\lambda}{d}+\epsilon.
\end{align*}

Finally, choosing $\epsilon$ small enough we get that $ -\frac{\lambda}{d}+\epsilon <0$ and then for big enough $t$, we get that
\[l_{n'}(t) \leq \left(-\frac{\lambda}{d}+\epsilon\right)  t+l_{n'}(0),\]
which means that $l_{n'}(t)$ converges to zero in finite time and the Ricci flow is not immortal.\\
\end{proof}

\section{Convergence results for awesome metrics}\label{section:5}

In this section we will further our long-time behavior analysis of awesome Ricci flows by examining the long-time limit solitons we obtain by approprietaly rescaling the Ricci flow solution $g(t)$. We first investigate the case where the universal cover of $G/H$ is not contractible, i.e. it is not diffeomorphic to $\mathbb{R}^n$ so that by Theorem \ref{theo:main} the extinction time is finite and later in Subsection \ref{subsection:5.2}, the contractible case, when $G/H$ is diffeomorphic to $\mathbb{R}^n$, where the flow is immortal. 

\subsection{The Non-Contractible Case} \label{subsection:5.1}

Let us consider the following formula for the scalar curvature $R(g)$ of an awesome metric $g$ which easily follows from the Ricci curvature formula \eqref{eq:nikonorov ricci}
\begin{equation}\label{eq:scalar awesome}
	R(g)= \sum_{i \in I_{\ll}}\frac{d_i}{2l_i}-\sum_{\hat{j} \in I_{\pp}}\frac{d_{\hat{j}}}{2p_j}-\sum_{i,j,k}[ijk]\frac{x_i}{x_jx_k}.
\end{equation}
We immediately see from the equation above that
\begin{equation}\label{eq:R < 1/l}
	R(g) \leq \frac{\dim \ll}{2l_1},
\end{equation}
which in turn means that if the scalar curvature blows-up, then the smallest eigenvalue in the $\ll$-direction goes to zero.

B\"ohm showed in \cite[Theorem 2.1]{boe} that every homogeneous Ricci flow with finite extinction time develops Type I singularity, namely, that there is a constant $C(g_0)$ that only depends on the initial metric $g_0$ such that we have the following upper bound to the norm of the Riemann tensor $\Rm(g)$ along the Ricci flow
\[\Vert \Rm(g(t))\Vert _{g(t)}\leq \frac{C(g_0)}{T-t},\]
for $t \in[0,T)$, where $T$ is the maximal time for the flow. Even more, if we assume $R(g_0)$ is positive, he showed that along a finite extinction time homogeneous Ricci flow the Riemann tensor is controlled by the scalar curvature \cite[Remark 2.2]{boe} and that there are constants $c(g_0)$ and $C(g_0)$ only depending on the initial metric $g_0$ such that
\[\frac{c(g_0)}{T-t}\leq R(g(t)) \leq\frac{C(g_0)}{T-t},\]
for $t \in [0,T)$.

This gives us a natural scaling parameter for a blow-up analysis of the Ricci flow solution $g(t)$. By \cite{emt}, we can extract a non-flat limit shrinking soliton from such a blow-up. In the following theorem, which is our second main result, we show that these limits only depend on the induced geometry on the compact fiber $K/H$.

\begin{theorem}
	Let $M=G/H$ be a semisimple homogeneous space such that the universal cover is not contractible. Let $(M,g(s))$, $s \in [0,T)$, be an awesome Ricci Flow adapted to the Cartan decomposition $\gg = \kk \oplus \pp$.  Let $R(g)$ be the scalar curvature of the metric $g$. For any sequence $(s_a)_{a\in\mathbb{N}}$, $s_a \to T$ there exists a subsequence such that $\left(M,R(g(s_{\hat{a}}))\cdot g(s_{\hat{a}})\right)$ converges in pointed $C^{\infty}$-topology to the Riemannian product
	\[E_{\infty} \times \mathbb{E}^d,\]
	where $E_{\infty}$ is a compact homogeneous Einstein manifold with positive scalar curvature and $\mathbb{E}^d$ is the $d$-dimensional (flat) Euclidean space with $d \geq \dim \pp$.
	
	Furthermore, the geometry of $E_{\infty}$ just depends on the subsequence of Riemannian submanifolds $\left(K/H,R(g(s_{\hat{a}}))\cdot g(s_{\hat{a}})\right)$.
\end{theorem}

\begin{proof}
	By Theorem \ref{theo:main} we know that such a solution $g(s)$ has finite time extinction, hence by \cite[Theorem 2.1]{boe} we know that it is a Type I flow. By \cite[Theorem 1.1]{laf} we can moreover assume without loss of generality that $R(g(0))>0$, so let us define the parabolic rescaled metric $\tilde{g}_s\coloneqq R(g(s))\cdot g(s)$. By the work of Enders, Müller, and Topping on Type I singularities of the Ricci flow \cite[Theorem 1.1]{emt}, it follows via Hamilton's compactness theorem \cite{ham} that along any sequence of times converging
	to the singularity time $T$, these scalar curvature normalized parabolic rescalings will subconverge to a non-flat homogeneous gradient shrinking soliton.
	
	Moreover, by work of Petersen and Wylie in \cite{pw}, we have that a homogeneous gradient shrinking soliton is rigid, in the sense that it is a Riemannian product of an Euclidean factor and a positive scalar curvature homogeneous Einstein manifold. 
	
	We have already seen in Equation \eqref{eq:r1} that for the direction of the smallest eigenvalue in the $\pp$-direction $p_1$, the Ricci curvautre is negative, $r^{\pp}_1<0$. Since any limit gradient shrinking soliton is Ricci non-negative, this implies that $\lim_{s\to T}\left(\frac{1}{R(g(s))}\cdot r_1^{\pp}(s)\right) = 0.$
	In particular, for any $\hat{j} \in I_{\pp}$ and $k \in I_{\ll}$,
	\[	\lim_{s\to T}\left([\hat{1}\hat{j}k]_s\left(\frac{p_{1}(s)}{p_j(s)}-\frac{p_j(s)}{p_{1}(s)}\right)\frac{1}{R(g(s))l_k(s)}\right)= 0.\]
	
	Which in turn implies that for the second largest eigenvalue in the $\pp$-direction $p_2$, $\lim r^{\pp}_2(s)\leq0$ as $t$ approaches the singular time $T$. Hence, once again we can conclude that for any $\hat{j} \in I_{\pp}$ and $k \in I_{\ll}$,
		\[	\lim_{s\to T}\left([\hat{2}\hat{j}k]_s\left(\frac{p_{2}(s)}{p_j(s)}-\frac{p_j(s)}{p_2(s)}\right)\frac{1}{R(g(s))l_k(s)}\right)= 0.\]
		
		Arguing recursively, we can then conclude that for any $\hat{i},\hat{j} \in I_{\pp}$ and $k \in I_{\ll}$, 
		
		\begin{equation}\label{eq:ricp IC to 0} 
	\lim_{s\to T}\left([\hat{i}\hat{j}k]_s\left(\frac{p_{i}(s)}{p_j(s)}-\frac{p_j(s)}{p_{i}(s)}\right)\frac{1}{R(g(s))l_k(s)}\right)= 0.
		\end{equation}
	Therefore, in the limit geometry we find at least as many linear independent directions as $\dim\pp$ such that the Ricci curvature is $0$. This can only be the case if the Euclidean factor in the limit has dimension at least  as large as $\dim \pp$. Furthermore, by \cite[Theorem 5.2]{boe} we know that this dimension does not depend on the subsequence taken.
	
	Observe that this in particular implies 
	\begin{align}\label{eq:a-b to 0}
\sum_{\substack{ \alpha,\beta,\gamma}}\left(\tilde{g}_s\left(\tilde{E}^{\pp_i}_{\alpha},\left[\tilde{E}^{\pp_j}_{\beta},\tilde{E}^{\ll_k}_{\gamma}\right]\right)
- Q\left(E^{\pp_i}_{\alpha},\left[E^{\pp_j}_{\beta},\tilde{E}^{\ll_k}_{\gamma}\right]\right)\right) \to 0, 
	\end{align}
where $\tilde{E}^{\pp_j}_\alpha=\frac{E^{\pp_j}_\alpha}{\sqrt{R(g(s))\cdot p_j(s)}}$ (respectively $\tilde{E}^{\ll_k}_\gamma=\frac{E^{\ll_k}_\gamma}{\sqrt{R(g(s))\cdot l_i(s)}}$) and $\{E^{\pp_j}_\alpha\}$ (respectively $\{E^{\ll_i}_\gamma\}$) is a $g(s)$-diagonalizing basis of $\pp$ (respectively of $\ll$) with respect to the $\Ad(K)$-invariant background metric $Q\coloneqq -B|_{\ll\times \ll} +B|_{\pp\times \pp}$.

Furthermore, let $X \in \ll$ and $x_s \coloneqq || X||^2_{\tilde{g}_s}$. The rescaled Ricci curvature on the compact Riemannian submanifold $N_s=(K/H,\tilde{g}_s)$ on the direction $\tilde{X}_s = \frac{X}{\sqrt{x_s}}$, as remarked in equation \eqref{eq:ricci tensor of in terms of the induced metric}, is given by

\begin{align*}
	\ric_{\tilde{g}_s}\left(\tilde{X}_s,\tilde{X}_s\right)=&\ \ric_{N_s}\left(\tilde{X}_s,\tilde{X}_s\right)
	-\frac{1}{2}\tr\left(\left.\text{ad}(\tilde{X}_s)\circ \text{ad}(\tilde{X}_s)\right|_{\pp}\right)\\
	&-\frac{1}{2}
	\sum_{\substack{\hat{i} \in I_{\pp} \\ \alpha}} \left \Vert\left[\tilde{X}_s,\tilde{E}^{\pp_i}_\alpha\right]\right \Vert_{\tilde{g}_s}^2+\frac{1}{4}
	\sum_{\substack{\hat{i},\hat{j} \in I_{\pp} \\ \alpha,\beta}} \tilde{g}_s\left([\tilde{E}^{\pp_i}_\alpha,\tilde{E}^{\pp_j}_\beta],\tilde{X}_s\right)^2.
\end{align*}

This means that

\begin{align*}
	\left|\ric_{\tilde{g}_s}(\tilde{X}_s,\tilde{X}_s)-\ric_{N_s}(\tilde{X}_s,\tilde{X}_s)\right| \leq&\ \frac{1}{2}\left(	\sum_{\substack{\hat{i},\hat{j} \in I_{\pp} \\ \alpha,\beta}}Q\left( E^{\pp_i}_\alpha, \left[E^{\pp_j}_\beta,\tilde{X}_s\right]\right)^2 
	-\sum_{\substack{\hat{i},\hat{j} \in I_{\pp} \\ \alpha,\beta}} \tilde{g}_s\left( \tilde{E}^{\pp_i}_\alpha, \left[\tilde{E}^{\pp_j}_\beta,\tilde{X}_s\right]\right)^2\right) \\  
	&+\frac{1}{4} \sum_{\substack{\hat{j},\hat{k} \in I_{\pp} \\ \alpha,\beta}}g_s\left( \left[\tilde{E}^{\pp_j}_\beta,\tilde{E}^{\pp_k}_\alpha\right],\tilde{X}_s\right)^2.
\end{align*}

For the first term on the right-hand side we have that

\begin{align*}
	 \tilde{g}_s\left( \tilde{E}^{\pp_i}_\alpha, \left[\tilde{E}^{\pp_j}_\beta,\tilde{X}_s\right]\right)^2 =\left(\sum_{\substack{k\in I_{\ll} \\ \gamma}}\tilde{g}_s\left(\tilde{X}_s,\tilde{E}^{\ll_k}_{\gamma}\right)\tilde{g}_s\left(\tilde{E}^{\pp_i}_{\alpha},\left[\tilde{E}^{\pp_j}_{\beta},\tilde{E}^{\ll_k}_{\gamma}\right]\right)\right)^2
\end{align*} 
and by what we observed in \eqref{eq:a-b to 0} this can be approximated by 
\[\left(\sum_{\substack{k\in I_{\ll} \\ \gamma}}\tilde{g}_s\left(\tilde{X}_s,\tilde{E}^{\ll_k}_{\gamma}\right)\tilde{g}_s\left(\tilde{E}^{\pp_i}_{\alpha},\left[\tilde{E}^{\pp_j}_{\beta},\tilde{E}^{\ll_k}_{\gamma}\right]\right)\right)^2
=\left(\sum_{\substack{k\in I_{\ll} \\ \gamma}}\tilde{g}_s\left(\tilde{X}_s,\tilde{E}^{\ll_k}_{\gamma}\right)Q\left(E^{\pp_i}_{\alpha},\left[E^{\pp_j}_{\beta},\tilde{E}^{\ll_k}_{\gamma}\right]\right)\right)^2+ \epsilon(s),
 \]
 where $\epsilon(s)=0$ as $s\to T$.

As for the second term on the right-hand side, observe that, by Lemma \ref{lemma:at most linear growth}, $l_n(s)$ grows at most linearly and $p_1(s)$ grows at least linearly, which implies that for any $\hat{i},\hat{j} \in I_\pp$ and $k\in I_\ll$,
\[\lim_{t\to T}\left(\frac{1}{R(g(s))}\cdot\frac{l_k(s)}{p_i(s)p_j(s)}\right) =0.\]
Hence, by Cauchy-Schwarz we get that 

\begin{align*}
	\tilde{g}_s\left( \left[\tilde{E}^{\pp_i}_\alpha,\tilde{E}^{\pp_j}_\beta\right],\tilde{X}_s\right)^2 &\leq  \left\Vert\left[\tilde{E}^{\pp_i}_\alpha,\tilde{E}^{\pp_j}_\beta\right] \right\Vert_{\tilde{g}_s}^2 =\sum_{\substack{k \in I_{\ll} \\ \gamma }} Q\left( \left[E^{\pp_i}_\alpha,E^{\pp_j}_\beta\right],E^{\ll_k}_{\gamma}\right)^2\cdot \frac{1}{R(g(s))}\cdot\frac{l_k(s)}{p_i(s)p_j(s)} \to 0.
\end{align*}

Therefore $\left|\ric_{\tilde{g}_s}(\tilde{X}_s,\tilde{X}_s)-\ric_{N_s}(\tilde{X}_s,\tilde{X}_s)\right| \to 0$ and we conclude that the Ricci $(1,1)$-tensor restricted to $\ll$ approximates the Ricci tensor given by the induced metric on $(K/H,\tilde{g}_s)$ as $s$ approaches the singularity time $T$.

Let us now consider the rescaled Ricci flow solution 
\[\tilde{g}_s(t):=R(g(s))\cdot g\left(s+\frac{t}{R(g(s))}\right)\]
restricted to the submanifold $K/H$. The argument above could be carried on taking by $\tilde{g}_s(t)$ instead of $\tilde{g}_s(0)$. So we have shown that the family of Ricci flow solutions $(K/H,\tilde{g}_s(t))$ is equivalent, as $s$ approaches the singularity time $T$, to the Ricci flow solution $(K/H,\hat{g}_s(t))$, where $\hat{g}_s(t)$ is the Ricci flow on $K/H$ with initial metric $\tilde{g}_s(0)|_{K/H}$. In particular, that means that the limit Einstein factor $E_\infty$ only depends on a convergent subsequence of the submanifold geometry of $\left(K/H,\tilde{g}_s(0)|_{K/H}\right)$. \\
\end{proof}

\subsection{The Contractible Case} \label{subsection:5.2}

For the sake of completeness, we want now to understand the limit geometry in the case when the universal cover of our semisimple homogeneous space is contractible.

A homogeneous Riemannian manifold $\tilde{M}$ diffeomorphic to $\mathbb{R}^n$ must be a Riemmanian product of a non-compact symmetric space and an $\mathbb{R}^d$-bundle over a \textit{Hermitian symmetric space} (see \cite[Proposition 3.1]{bl22}). Hermitian symmetric spaces are a special class of non-compact symmetric spaces which are also Hermitian manifolds. Irreducible ones correspond to irreducible symmetric pairs $(\gg,\kk)$ with $\dim \zz(\kk)=1$. In particular, if we write $\kk=\kk_{ss}\oplus\zz(\kk)$ and consider the integral subgroup $K_{ss}\subset G$ with Lie algebra $\kk_{ss}$, then the homogeneous space $G/K_{ss}$ is a homogeneous line bundle over the irreducible Hermitian symmetric space $G/K$ (for more on Hermitian symmetric spaces cf. \cite[Chapter VIII, Theorem 6.1]{hel}, see also \cite[7.104]{be}).

An example is the product $\reallywidetilde{\left(SL(2,\mathbb{R})\right)^k}$ with left-invariant metrics, which is an $\mathbb{R}^k$-bundle over the product of hyperbolic planes $\left(\mathbb{H}^{2}\right)^k$. It is worth mentioning that semisimple homogeneous $\mathbb{R}$-bundles over irreducible Hermitian symmetric spaces only admit awesome homogeneous metrics (see \cite[Remark 3.2]{bl22}), which is the case for example of $\reallywidetilde{SL(2,\mathbb{R})}$.

Let $\gg$ be semisimple of non-compact type with Cartan decomposition $\gg = \kk \oplus \pp$ and $\hh$ a compactly embededd subalgebra of $\kk$. Let $G$, $K$ and $H$ be Lie groups for the Lie algebras $\gg$, $\kk$, and $\hh$ respectively such that $G/H$ is simply-connected. We know that $G/H = K/H \times \mathbb{R}^{\dim \pp}$ is diffeomorphic to $\mathbb{R}^n$ if and only if $[\kk,\kk] \subset \hh$. Given the reductive decomposition $\mm = \ll \oplus \pp $, where $\gg = \hh \oplus \mm$, this is equivalent to 
\[[\hh,\ll] =0 \hspace{0.5cm} \text{and} \hspace{0.5cm} [\ll,\ll] \subset \hh.\]

This implies that for any $G$-invariant metric $g$ on $G/H$, $(K/H, \left.g\right|_{K/H})$ is isometric to a flat Euclidean space. If $g$ is awesome, then by Remark \ref{remark:toral case} this implies that
\[r^{\ll}_i = \frac{1}{4d_i}\sum_{\hat{j},\hat{k} \in I_{\pp}}[i\hat{j}\hat{k}]\left(\frac{2}{l_i} + \frac{l_i}{p_j p_k}-\frac{p_j}{l_i p_k}-\frac{p_k}{l_ip_j}\right).\]

We already know, by Corollary \ref{corollary:p/p goes to 1}, that on an immortal awesome Ricci flow, $\left.g(t)\right|_{\pp \times \pp}$ approximates $t\cdot \left.B\right|_{\pp \times \pp}$ and, by Corollary \ref{corollary:g_s goes to 0}, that the blow-down of $g(t)$ in the $\ll$-direction goes to 0. These estimates are enough to have the following convergence result, which is the third main result of this article.

\begin{theorem}\label{theo:conv immortal}
		Let $\tilde{M}=G/H$ be a contractible semisimple homogeneous space. Let $(\tilde{M},g(s))$, $s \in [1,\infty)$, be an immortal awesome Ricci Flow adapted to the Cartan decomposition $\gg = \kk \oplus \pp$, with $\kk = \hh\oplus\ll$. Then the parabolic rescaling $(\tilde{M},s^{-1}g(s))$ converges in pointed $C^{\infty}$-topology to the Riemannian product
	\[ \Sigma_{\infty}\times \mathbb{E}^{\dim \ll},\]
	where $\Sigma_{\infty}$ is the non-compact Einstein symmetric space $\left(G/K,\left.B\right|_{\pp\times\pp}\right)$ and $\mathbb{E}^d$ is the $d$-dimensional (flat) Euclidean space.
\end{theorem}

\begin{proof}
	Let us fix the background metric $Q \coloneqq -B|_{\ll\times \ll} +B|_{\pp\times \pp}$ and let $g(s) = Q(P(s)\cdot,\cdot)$. Let us denote $\tilde{g}_s \coloneqq s^{-1}g(s)$. Since $[\kk,\kk]\subset \hh$ without loss of generality we can assume $\ll \subset \zz(\kk)$ (just take $\ll \perp_B \hh$). By \cite[Theorem 13.1.7]{hn} we have the following diffeomorphism
	
		\begin{align*}
		\phi \colon  \pp \times K/H  &\xrightarrow{\sim} G/H \\
		(x, kH) &\mapsto \exp(x)\cdot kH.
	\end{align*}
Moreover,  since $K/H=\mathbb{R}^{\dim \ll}$ is an abelian group, then $K/H = \exp(\ll)H$. 

Let us then define the following 1-parameter family of diffeomorphisms

	\begin{align*}
	\phi_s \colon  \pp \times \ll  &\xrightarrow{\sim} G/H \\
	(x, u) &\mapsto \alpha(x)\cdot \beta_s(u)H.
\end{align*}
with $\alpha(x)=\exp(x)$ and $\beta_s(u)=\exp(\sqrt{s(P{^\ll}(s))^{-1}}u)$, where $P^{\ll}(s)$ is the positive definite matrix defined by $\left. g(s)\right|_{\ll\times\ll}=Q(P^{\ll}(s)\cdot,\cdot)$.

Let $\frac{\partial \phi_s }{\partial E^\ll_i}(x,u) = \left. \frac{d}{dt}\alpha(x)\cdot\beta_s(u+tE^\ll_i)H\right|_{t=0}$ and notice that since $K/H$ is abelian, we get that

	\begin{align*}
	\left(L_{\left(\alpha(x)\cdot\beta_s(u)\right)^{-1}}\right)_*\frac{\partial \phi_s}{\partial E^\ll_i}(x,u) &= \left(L_{\beta_s(-u)}\right)_*\left(L_{\alpha(-x)}\right)_* \left(L_{\alpha(x)}\right)_*\left. \frac{d}{dt}\beta_s(u+tE^\ll_i)H \right|_{t=0}\\
	&=\left(L_{\beta_s(-u)}\right)_* \left. \frac{d}{dt}\beta_s(u+tE^\ll_i)H\right|_{t=0} \\
	&=\left. \frac{d}{dt}\beta_s(tE^\ll_i)H\right|_{t=0}.
\end{align*}
This implies the following

	\begin{align*}
	\phi_s^*\tilde{g}_s(E^\ll_i,E^\ll_j)_{(x,u)} &= \tilde{g}_s\left(\frac{\partial \phi_s^\beta }{\partial i}(x,u),\frac{\partial \phi_s^\beta }{\partial j}(x,u)\right)_{\phi_s(x,u)} \\
	&= \tilde{g}_s \left(\left. \frac{d}{dt}\beta_s(tE^\ll_i)\right|_{t=0},\left. \frac{d}{dt}\beta_s(tE^\ll_i)\right|_{t=0}\right)_{\phi_s(0,0)}\\
	&=s^{-1}Q \left(P^{\ll}(s) \cdot \sqrt{s\left(P^{\ll}(s)\right)^{-1}}E_i^\ll, \sqrt{s\left(P^{\ll}(s)\right)^{-1}}E_j^\ll\right) \\
	&= Q(E_i^\ll,E_j ^\ll).
\end{align*}

Let $\frac{\partial \phi_s}{\partial E^\pp_i}(x,u) = \left. \frac{d}{dt}\alpha(x+tE_i^\pp)\cdot\beta_s(u)\right|_{t=0}$ and let $\mathcal{A}_s^i(x,u) \coloneqq \left(L_{\left(\alpha(x)\cdot\beta_s(u)\right)^{-1}}\right)_*\frac{\partial \phi_s }{\partial E^\pp_i}(x,u)$. 

Observe that since $K$ belongs to the normalizer of $H$ in $G$, it acts on the right on $G/H$. Hence, using that $a^{-1}\exp(x) a = \exp\left(\Ad(a)x\right)$ \cite[Proposition 9.2.10]{hn}, we get that

\begin{align*}
	\mathcal{A}_s^i(x,u) &=  \left(L_{\beta_s(-u)}\right)_*\left(L_{\alpha(-x)}\right)_* \left(R_{\beta_s(u)}\right)_*\left. \frac{d}{dt}\alpha(x+tE_i^\pp )\right|_{t=0} \\
	&= \left(L_{\beta_s(-u)}\right)_*\left(R_{\beta_s(u)}\right)_*\left(L_{\alpha(-x)}\right)_*\left. \frac{d}{dt}\alpha(x+tE_i^\pp )\right|_{t=0} \\ 
	&=\Ad(\beta_s(u))\left(L_{\exp\left(-x\right)}\right)_*\left.\frac{d\exp }{dt}\left(x+tE^\pp_i\right)\right|_{t=0}
\end{align*}
and that $\mathcal{A}_s^i(x,0) = \mathcal{A}^i(x)$ does not depend on $s$.

Now observe that by \cite[Theorem 13.1.5]{hn} $\Ad(K)$ is compact. Hence, there is a constant $C$, such that 
\[\left\Vert \Ad(\beta_s(u)) \right\Vert_Q \leq  C,\]
 in the operator norm with respect to $Q$. By Corollary \ref{corollary:g_s goes to 0}, we have that
 
	\begin{align*}
	\left\vert \phi_s^*\tilde{g}_s(E^\pp_i,E^\ll_j)_{(x,u)} \right\vert&= \left\vert \tilde{g}_s\left(\frac{\partial \phi_s }{\partial E^\pp_i}(x,u),\frac{\partial \phi_s }{\partial E^\ll_j}(x,u)\right)\right\vert_{\phi_s(x,u)} \\
	&= \left\vert\tilde{g}_s \left(\mathcal{A}_s^i(x,u),\left. \frac{d}{dt}\beta_s(tE^\ll_j)\right|_{t=0}\right)\right\vert_{\phi_s(0,0)}\\
	&=s^{-1}\left\vert Q \left(\mathcal{A}_s^i(x,u), P^{\ll}(s) \cdot \sqrt{s\left(P^{\ll}(s)\right)^{-1}}E_j^\ll\right)\right\vert \\
	&= \sqrt{s^{-1}}\left\vert Q \left(\mathcal{A}_s^i(x,u),  \sqrt{P^{\ll}(s)}E_j^\ll\right)\right\vert \\
	&\leq C \sqrt{\frac{\left\Vert P^\ll(s)\right\Vert_{Q}}{s}}\cdot \left\Vert \mathcal{A}^i(x)\right\Vert_Q  \left \Vert E_j^\ll\right\Vert_Q \to 0,
\end{align*}
uniformly on compact sets of $\pp \times \ll$.

Finally, we have that

\begin{align*}
	\phi_s^*\tilde{g}_s(E_i^\pp,E_j^\pp)_{(x,u)} &= \tilde{g}_s\left(\frac{\partial \phi_s }{\partial E_i^\pp}(x,u),\frac{\partial \phi_s}{\partial E_j^\pp}(x,u)\right)_{\phi_s(x,u)} \\
	&= \tilde{g}_s \left(\mathcal{A}_s^i(x,u),\mathcal{A}_s^j(x,u)\right)_{\phi_s(0,0)} \\
	&=\left.\tilde{g}_s\right|_{\ll\times\ll} \left(\mathcal{A}_s^i(x,u),\mathcal{A}_s^j(x,u)\right)+\left.\tilde{g}_s\right|_{\pp\times\pp}  \left(\mathcal{A}_s^i(x,u),\mathcal{A}_s^j(x,u)\right).
\end{align*} 
Again by the fact that $\Ad(K)$ is compact and by Corollary \ref{corollary:g_s goes to 0}, we have that
\[\left\vert\left.\tilde{g}_s \right|_{\ll\times\ll} \left(\mathcal{A}_s^i(x,u),\mathcal{A}_s^j(x,u)\right)\right\vert\leq C^2\frac{\left\Vert P^\ll(s)\right\Vert_{Q}}{s}\cdot\left\Vert \mathcal{A}^i(x)\right\Vert_Q \left\Vert \mathcal{A}^j(x) \right\Vert _Q\to 0,\]
uniformly on compact subsets of $\pp\times \ll$.
Moreover, by Corollary \ref{corollary:p/p goes to 1} we know that $\tilde{g}_s|_{\pp\times\pp} \to B|_{\pp\times\pp}$ and since $B$ is $\Ad(K)$-invariant, we get that

\begin{align*}
	\left.\tilde{g}_s \right|_{\pp\times\pp}  \left(\mathcal{A}_s^i(x,u),\mathcal{A}_s^j(x,u)\right) \to \left.Q\right|_{\pp\times\pp}  \left(\mathcal{A}^i(x),\mathcal{A}^j(x)\right),
\end{align*}
uniformly on compact subsets of $\pp\times \ll$. Indeed, assume the contrary, then there exists $\epsilon >0$ and sequences $s_n \to \infty$ and $(x_n,u_n) \to (x_\infty,u_\infty)$ such that 

\begin{equation}\label{eq:convergence to symmetric}
	\left \vert\left.\tilde{g}_{s_n} \right|_{\pp\times\pp}  \left(\mathcal{A}_{s_n}^i(x_n,u_n),\mathcal{A}_{s_n}^j(x_n,u_n)\right) - \left.Q\right|_{\pp\times\pp}  \left(\mathcal{A}^i(x_n),\mathcal{A}^j(x_n)\right)\right\vert \geq \epsilon.
\end{equation}
By the compactness of $\Ad(K/H)$, we can extract a convergent subsequence, $\Ad(\beta_{s_n}(u_n)) \to \Ad\left(\exp(u'_\infty)\right)$, $u'_\infty \in \kk$. Therefore, taking the limit on \eqref{eq:convergence to symmetric} as $n \to \infty$, we get that

\begin{align*}
	\left \vert\left.Q \right|_{\pp\times\pp}  \left(\Ad\left(\exp(u'_\infty)\right)\cdot \mathcal{A}^i(x_\infty),\Ad\left(\exp(u'_\infty)\right)\cdot\mathcal{A}^j(x_\infty)\right) - \left.Q\right|_{\pp\times\pp}  \left(\mathcal{A}^i(x_\infty),\mathcal{A}^j(x_\infty)\right)\right\vert =0,
\end{align*}
since $Q$ is $\Ad(K)$-invariant. 

Observe that $\left.Q\right|_{\pp\times\pp}  \left(\mathcal{A}^i(x),\mathcal{A}^j(x)\right)$ is the pullback by the diffeomorphism
\[\exp \colon \pp \to G/K\]
 of the Ricci negative Einstein symmetric metric of $\Sigma_{\infty}\coloneqq \left(G/K,\left.B\right|_{\pp\times\pp}\right)$. 
  
Hence, proving that $\phi_s^*\tilde{g}_s$ converges in the $C^{\infty}$-topology to the Riemannian product $\Sigma_{\infty}\times\mathbb{E}^{\dim \ll}$.
\end{proof}

\begin{remark*}
	Since every immortal homogeneous Ricci flow is of Type III \cite[Theorem 4.1]{boe}, the proof above actually shows that the parabolic rescaled flow $\tilde{g}_s(t) \coloneqq s^{-1}g(st)$, $t \in (0,\infty)$, converges in Cheeger-Gromov sense to the expanding Ricci soliton given by the Riemannian product of the Ricci negative Einstein metric in $\Sigma_{\infty}$ and the flat Euclidean factor $\mathbb{E}^{\dim\ll}$. This generalizes the 3-dimensional case $\reallywidetilde{SL(2,\mathbb{R})} \to \mathbb{H}^2 \times \mathbb{R}$ (see \cite[Case 3.3.5]{lot}) to every semisimple homogeneous $\mathbb{R}$–bundle
	over irreducible Hermitian symmetric spaces $G/H$, since for those every $G$-invariant metric is awesome \cite[Remark 3.2]{bl22}.
\end{remark*}

\end{document}